\documentclass{amsart}

\usepackage{hyperref}
\usepackage{epsfig}
\usepackage{appendix}
\usepackage{amsfonts,amssymb,amsmath,amsthm,amscd}
\usepackage{latexsym}
\usepackage{mathtools}
\usepackage{graphicx}
\usepackage[all, cmtip]{xy}

\theoremstyle{plain}
		\newtheorem{theorem}{Theorem}[section]
		\newtheorem{lemma}[theorem]{Lemma}
		\newtheorem{corollary}[theorem]{Corollary}
		\newtheorem{proposition}[theorem]{Proposition}

		\newtheorem{fact}[theorem]{Fact}
\theoremstyle{definition}
		\newtheorem{definition}[theorem]{Definition}

\theoremstyle{remark}
		\newtheorem*{remark}{Remark}
		\newtheorem{example}[theorem]{Example}
		
\newcommand{\A}{{\mathbb{A}}}

\newcommand{\C}{{\mathbb{C}}}

\newcommand{\G}{{\mathbb{G}}}

\newcommand{\M}{{\mathbb{M}}}

\renewcommand{\P}{{\mathbb{P}}}
\newcommand{\Q}{{\mathbb{Q}}}
\newcommand{\R}{{\mathbb{R}}}

\newcommand{\Z}{{\mathbb{Z}}}

\newcommand{\Cb}{{\mathbf{C}}}
\newcommand{\Db}{{\mathbf{D}}}

\newcommand{\Ib}{{\mathbf{I}}}

\newcommand{\Lb}{{\mathbf{L}}}

\newcommand{\Ocal}{{\mathcal{O}}}

\newcommand{\pfrak}{{\mathfrak{p}}}

\newcommand{\Cfrak}{{\mathfrak{C}}}

\newcommand{\Char}{\textup{Char}}

\newcommand{\Crit}{\textup{Crit}}

\newcommand{\del}{\partial}

\DeclareMathOperator{\Div}{Div}

\DeclareMathOperator{\Fil}{Fil}
\DeclareMathOperator{\Fin}{Fin}

\DeclareMathOperator{\Gr}{Gr}

\DeclareMathOperator{\Hom}{Hom}

\DeclareMathOperator{\sheafhom}{\mathcal{H}\kern -.5pt \emph{om}}
\DeclareMathOperator{\Img}{Im}

\DeclareMathOperator{\length}{length}

\DeclareMathOperator{\modulo}{\:\textup{mod}\:}

\DeclareMathOperator{\Proj}{Proj}

\DeclareMathOperator{\Rep}{Rep}

\DeclareMathOperator{\SL}{SL}

\DeclareMathOperator{\Span}{Span}
\DeclareMathOperator{\Spec}{Spec}

\DeclareMathOperator{\tr}{tr}

\newcommand{\git}{\mathbin{
  \mathchoice{/\mkern-6mu/}
    {/\mkern-6mu/}
    {/\mkern-5mu/}
    {/\mkern-5mu/}}}

\begin{document}
\title[Global geometry on moduli of local systems]{Global geometry on moduli of local systems\\for surfaces with boundary}
\author{Junho Peter Whang}
\email{jwhang@mit.edu}
\address{Massachusetts Institute of Technology, Department of Mathematics, Simons Building Room 2-238A, 77 Massachusetts Avenue, Cambridge, MA 02139-4307}



\begin{abstract}
We show that every coarse moduli space, parametrizing complex special linear rank two local systems with fixed boundary traces on a surface with nonempty boundary, is log Calabi-Yau in that it has a normal projective compactification with trivial log canonical divisor. We connect this to a novel symmetry of generating series for counts of essential multicurves on the surface.
\end{abstract}

\maketitle

\section{Introduction} \label{sect:1}

\subsection{Main results} \label{sect:1.1}
Let $\Sigma_{g,n}$ be a compact oriented surface of genus $g$ with $n\geq1$ boundary curves satisfying $3g+n-3>0$. Let $X_{g,n}$ be the coarse moduli space of $\SL_2(\C)$-local systems on $\Sigma_{g,n}$, also called the (full) character variety. Let $X_{g,n,k}$ be the subvariety of $X_{g,n}$ obtained by prescribing the boundary traces $k\in\A^n(\C)$. Our main result is that each $X_{g,n,k}$ is log Calabi-Yau in the following sense.
\begin{theorem}\label{mtheorem}
Each $X_{g,n,k}$ has a normal irreducible projective compactification $Z$ with canonical divisor $K_Z$ and reduced boundary divisor $D$ satisfying $K_Z+D\sim0$.
\end{theorem}
In particular, $X_{g,n,k}$ is normal and irreducible with trivial canonical divisor. This confirms a special case of the folklore conjecture that (relative) character varieties of surfaces are log Calabi-Yau; see remarks below. It is also closely related to the following combinatorial result about curves on surfaces. Let us recall the standard presentation of the fundamental group
\begin{align}
\label{sp}
\pi_1(\Sigma_{g,n})=\langle a_1,\dots,a_{2g+n}|[a_1,a_2]\dotsm[a_{2g-1},a_{2g}]a_{2g+1}\dotsm a_{2g+n}\rangle.
\end{align}
Since $n\geq1$ by assumption, $\pi_1(\Sigma_{g,n})$ is freely generated by the collection of simple loops $\sigma=\{a_1,\dots,a_{2g+n-1}\}$. Given a simple closed curve $a\subset\Sigma_{g,n}$, let $\length_{\sigma}(a)$ be the minimum $\sigma$-word length of any element $b\in \pi_1(\Sigma_{g,n})$ freely homotopic to a parametrization of $a$. This notion extends additively to \emph{multicurves} on $\Sigma_{g,n}$, which are finite disjoint unions of simple closed curves on $\Sigma_{g,n}$. We shall say that a multicurve is \emph{nondegenerate} if none of its components is contractible, and is \emph{essential} if it is nondegenerate and none of its components is isotopic to a boundary curve of $\Sigma_{g,n}$. Consider the formal power series $Z_{g,n}(t)=\sum_{r=0}^\infty c_{g,n}(r)t^r$ where $c_{g,n}(r)$ denotes the number of isotopy classes of essential multicurves $Q\subset \Sigma_{g,n}$ with $\length_\sigma(Q)=r$.

\begin{theorem}
\label{mtheorem2}
The series $Z_{g,n}(t)$ is rational, and satisfies the symmetry
$$Z_{g,n}(1/t)=Z_{g,n}(t).$$
\end{theorem}

This gives a topological interpretation of the log Calabi-Yau property of the moduli spaces $X_{g,n,k}$. Theorems \ref{mtheorem} and \ref{mtheorem2} are related by a remarkable result of Charles and March\'e \cite{cm}, which states that the isotopy classes of nondegenerate multicurves in $\Sigma_{g,n}$ provide a $\C$-linear basis for the coordinate ring of $X_{g,n}$ through their associated trace functions.

\subsection{Outline of proofs}
\label{sect:1.2}
We now describe the proofs of the above theorems and the contents of this paper. Our proof of Theorem \ref{mtheorem} proceeds in three steps: (i) we construct a graded algebra whose $\Proj$ is a compactification of $X_{g,n,k}$, (ii) we establish its algebraic properties, and (iii) we convert these properties to geometric information about the projective variety. This graded algebra is constructed (with the choice $\sigma$ of generators of $\pi_1(\Sigma_{g,n})$ above) using results from Section \ref{sect:2}, where we introduce the notion of \emph{word compactification} for the $\SL_2(\C)$-character variety of an arbitrary group with respect to a finite set of generators.

In Section \ref{sect:3}, we show that the algebra constructed using Section \ref{sect:2} is a graded Gorenstein domain with an explicit canonical module, which (after applying works of Hochtster-Roberts \cite{hr} and Stanley \cite{stanley}) amounts to a certain functional equation for the Hilbert series of the graded algebra. This functional equation is in turn deduced from an elementary invariant theoretic work of Le Bruyn \cite{lebruyn}. An analysis of singularities of the algebra (given separately in Section \ref{sect:4}) shows that the algebra is moreover normal. This allows us invoke results of Demazure \cite{demazure} and Watanabe \cite{watanabe} on graded normal rings to convert these algebraic properties to the geometry of our compactification, thereby proving Theorem \ref{mtheorem}.

In Section \ref{sect:5}, we use the theorem of Charles-March\'e \cite{cm} to establish an equivalence between the functional equation obtained in Section \ref{sect:3} and Theorem \ref{mtheorem2}, proving the latter. In Section \ref{sect:6}, we give another, purely combinatorial proof of Theorem \ref{mtheorem2} in the case $\Sigma_{g,n}$ has genus zero, where the main input comes from properties of the associahedron. Appendix A contains elementary computations in matrices relevant for the analysis in Section \ref{sect:4}.

\subsection{Example}
\label{sect:1.3}
We illustrate our results in the case $(g,n)=(1,1)$ of a one holed torus. (Further examples are given in Example \ref{examples}.) Let $\sigma=\{a_1,a_2\}$ be the free generators of $\pi_1(\Sigma_{1,1})$ as in the notation of (\ref{sp}), so that $a_1$ and $a_2$ correspond to simple loops on $\Sigma_{1,1}$ which intersect each other once transversally at the base point. The trace functions give us an identification
$$(\tr_{a_1},\tr_{a_2},\tr_{a_1a_2}):X_{1,1}\simeq\A_{x,y,z}^3$$
due to Fricke, under which the moduli space $X_{1,1,k}$ is presented as an affine surface with equation
$$x^2+y^2+z^2-xyz-2=k.$$
Our compactification of $X_{1,1,k}$ is precisely the projective closure $Z$ of $X_{1,1,k}$ in the weighted projective space $\P(1,1,1,2)$, where the variables $x$ and $y$ are given degree $1$ and $z$ is given degree $2$. Since $Z$ is a degree $4$ hypersurface in $\P(1,1,1,2)$, applying the adjunction formula we find that $K_Z+D\sim0$. Note that the Hilbert series $H_D(t)$ of the boundary divisor $D$ is
$$H_D(t)=\frac{(1-t)(1-t^4)}{(1-t)^3(1-t^2)}=\frac{1+t^2}{(1-t)^2}.$$
On the other hand, using train tracks (see examples in Figure \ref{fig1}), one can see that an essential multicurve on $\Sigma_{1,1}$ is uniquely determined by a pair $(a,b)$ of integers lying in $\{(0,b):b\in\Z_{\geq0}\}\cup\{(a,b):a\in\Z_{\geq1},b\in\Z\}$, and that $\length_\sigma(a,b)=a+|b|$ under this identification. In particular, we compute
$$Z_{1,1}(t)=\sum_{b=0}^\infty t^b+\sum_{a=1}^\infty\sum_{b\in\Z}t^{a+|b|}=\frac{1}{1-t}+\frac{t}{1-t}\frac{1+t}{1-t}=\frac{1+t^2}{(1-t)^2}$$
so that $Z_{1,1}(1/t)=Z_{1,1}(t)$. Note that $Z_{1,1}(t)=H_D(t)$.

\begin{figure}[ht]
    \centering
    \includegraphics{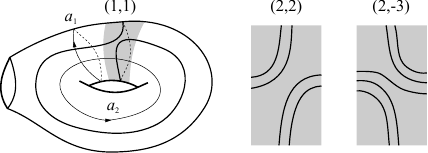}
    \caption{Some essential multicurves on $\Sigma_{1,1}$}
    \label{fig1}
\end{figure}

\subsection{Remarks}
\label{sect:1.4}
For $k$ integral, the varieties $X_{g,n,k}$ admit natural integral models. Our main motivation for Theorem \ref{mtheorem} comes from its direct ramifications for the expected Diophantine analysis of the integral points. For instance, in \cite{whang2}, we show that the set of integral points $X_{g,n,k}(\Z)$ is finitely generated, in a suitable sense, under the natural action of the mapping class group. This invites an analogy with finite generation results for integral points on log Calabi-Yau varieties of linear type, such as algebraic tori and abelian varieties. This analogy is explored further in \cite{whang4}, where we give an effective Diophantine analysis of integral points for nondegenerate algebraic curves on $X_{g,n,k}$, partly inspired by works of Vojta \cite{vojta2}, \cite{vojta3} and Faltings \cite{faltings} on subvarieties of semiabelian varieties.

The varieties $X_{g,n,k}$ are in general singular. The definition of a possibly singular log Calabi-Yau pair $(Z,D)$ varies somewhat across the literature (see e.g.~\cite{kollar}, \cite{ck}, \cite{kx}, \cite{ghk1}), in the type of boundary divisor $D$, type of triviality of the log canonical divisor $K_Z+D$, and singularity type. Our notion coincides with that of \cite{ck}.

Admitting the normality of $X_{g,n,k}$ proved in this paper, the triviality of the canonical divisor of the affine variety $X_{g,n,k}$ can also be deduced from the fact that the smooth locus $X_{g,n,k}^s\subset X_{g,n,k}$ admits an algebraic symplectic form, by work of Goldman (see for example \cite{goldman3}). The top power of the symplectic form provides a holomorphic volume form which trivializes the canonical bundle on $X_{g,n,k}^s$. Another construction of a volume form is given by the Reidemeister torsion, see \cite{porti}. Informally, Theorem \ref{mtheorem} can be viewed as saying that such a holomorphic volume form on $X_{g,n,k}$ has simple poles at infinity for the compactification at hand.

Given the explicit nature of our compactifications, the results of this paper may be relevant to conjectures surrounding the boundary divisors of relative character varieties for surfaces, as formulated by Simpson \cite{simpson} (verified therein for punctured spheres; see also \cite{komyo}), or more generally of log Calabi-Yau varieties, cf.~\cite{kx}. We leave this interesting topic for a future investigation.

For the ``decorated'' moduli spaces of local systems of Fock and Goncharov \cite{fg} which are related to, but different from, the moduli spaces considered in this paper, log Calabi-Yau compactifications of open subspaces can be constructed by general techniques of cluster varieties established by Gross-Hacking-Keel \cite{ghk1}; see also \cite{ghkk} for an elaboration of this point.

We mention another interesting compactification of the $\SL_2$ (or more general) character variety for free groups due to Manon \cite{manon}, using different Rees algebras associated to trivalent graphs, arising in the context of moduli spaces of principal bundles on closed Riemann surfaces and their degenerations. (Actually, in the setting of free groups with free generators, our word compactification coincides with a special case of a construction considered in \cite{manon}; see our remark at the end of Section \ref{sect:2.4}.) See also \cite{manon2} for another compactification.

\subsection{Conventions}
\label{sect:1.5}
Throughout this paper, a ring or an algebra will always mean a commutative algebra with unity over the complex numbers $\C$, unless otherwise specified. All schemes will be defined over $\C$. Given an affine scheme $X$, we denote by $\C[X]$ its coordinate ring. Let $\M$ denote the affine scheme parametrizing $2\times2$ matrices. Let $\{x_{ij}\}_{i,j\in\{1,2\}}$ be the regular functions on $\M$ corresponding to the $(i,j)$th entries of a matrix, giving an identification $(x_{ij}):\M\simeq\A^4$ with affine space. The standard matrix variable $x$ and its adjugate $x^*$ are the $2\times2$ matrices with coefficients in $\C[\M]$ given by
$$x=\begin{bmatrix}x_{11} & x_{12}\\ x_{21} & x_{22}\end{bmatrix},\quad \text{and}\quad x^*=\begin{bmatrix}x_{22}& -x_{12}\\ -x_{21} & x_{11}\end{bmatrix}.$$
In particular, $\det(x)$ and $\tr(x)$ are regular functions on $\M$. Let ${}^*:\M\to\M$ denote the involution on $\M$ determined by $x\mapsto x^*$.

Given $2\times2$ matrices $a$ and $b$, we shall denote $\langle a,b\rangle=aba^*b^*$ and $[a,b]=ab-ba$ in this paper, to avoid confusion. Given complex numbers $t,k\in\C$, let $\M_t\subset\M$ be the subscheme of matrices with determinant $t$, and let $\M_{t,k}\subset\M_t$ be the subscheme of matrices with determinant $t$ and trace $k$. We shall denote $\SL_2=\M_{1}$ and $\SL_{2,k}=\M_{1,k}$. Let $\mathbf1$ be the identity matrix. We shall say that a matrix $a$ is \emph{scalar} if $a=\lambda\mathbf1$ for some $\lambda\in\C$, and \emph{nonscalar} otherwise. As usual, let $\mathfrak{sl}_2(\C)$ denote the tangent space to $\SL_2$ at $\mathbf1$. It is the space of traceless matrices in $\M_2(\C)$.

Given an integer $m\geq1$, consider $\Cfrak_m=\Z/m\Z$ with its natural cyclic ordering. By a cyclic interval in $\Cfrak_m$ we shall mean a sequence of the form $i_0,i_0+1,\dots,i_0+k$ for some $i_0\in \Cfrak_m$ and for some $k\in\{0,\dots,m-1\}$. If $|I|=k+1<m$, then those elements of $\Cfrak_m$ not contained in $I$ together form a cyclic interval, denoted $I^c$.

\subsection{Acknowledgements}
\label{sect:1.6}
This work was done as part of the author's Ph.D.~thesis at Princeton University. I thank my advisor Peter Sarnak, Phillip Griffiths, and Sophie Morel for their guidance and encouragement during the writing of this paper. I also thank Carlos Simpson and Sean Keel for correspondences illuminating the literature, and Yuchen Liu for useful conversations. I thank Julien March\'e and an anonymous referee for providing helpful remarks and suggesting improvements to this paper.

\section{Word compactifications}
\label{sect:2}

In this section, we introduce the notion of word compactification for the character variety of a group with respect to a finite set of generators. In Sections \ref{sect:2.1} and \ref{sect:2.2}, we recall basic facts about graded algebras and Rees algebras, mainly to set our notation. We introduce the ($\SL_2$) character varieties of groups and their basic facts in Section \ref{sect:2.3}. In Section \ref{sect:2.4}, we define the notion of word compactification for a character variety, and give an alternate description in Section \ref{sect:2.5}.

\subsection{Graded rings} \label{sect:2.1}

Given an affine scheme $X$ with an action by the multiplicative group $\G_m$, its coordinate ring $A=\C[X]$ also carries a $\G_m$-action, or equivalently the structure of a graded ring via the $\G_m$-eigenspace decomposition:
$$A=\bigoplus_{r\in\Z}A_r,\quad \text{where}\quad A_r=\{a\in A:\lambda\cdot a=\lambda^{-r}a\text{ for all }\lambda\in\G_m(\C)\}.$$
Conversely, the spectrum $\Spec A$ of a graded ring $A$ carries a natural $\G_m$-action. Given a morphism of graded rings $B\to A$, we have an associated $\G_m$-equivariant morphism of schemes $\Spec A\to\Spec B$, and given $b\in\Spec B(\C)$ and $\lambda\in\G_m(\C)$ the fibers $(\Spec A)_b$ and $(\Spec A)_{\lambda\cdot b}$ are isomorphic via $\lambda$.

\begin{example}
For the grading of the polynomial ring $\C[t]$ by degree, the associated $\G_m$-action on the affine line $\A^1=\Spec\C[t]$ is the usual one $\lambda\cdot a=\lambda a$ on points.
\end{example}

Let $A$ be a graded algebra. An element $t\in A_1$ is non-nilpotent if and only if the morphism $\C[t]\to A$ is injective, or in other words the associated $\G_m$-equivariant morphism of schemes $\pi: X=\Spec A\to\A^1$ is dominant. Let $X_0=\Spec A/(t)$ and $X_1=\Spec A/(t-1)$ be the fibers of $\pi$ over $0$ and $1$, respectively, and let $X|_{\G_m}=\pi^{-1}(\G_m)$ where $\G_m=\Spec\C[t,1/t]\subset\A^1$ is the complement of $0$.

\begin{lemma}
\label{2l1}
Let $A$ be a graded ring, and let $t\in A_1$ be non-nilpotent as above. We have a natural $\G_m$-equivariant isomorphism
$$X_1\times\G_m\simeq X|_{\G_m},$$
equipping $X_1$ with the trivial $\G_m$-action. In particular, $A/(t-1)\simeq A[1/t]_0$.
\end{lemma}

\begin{proof}
The morphisms $X_1\times\G_m\to X|_{\G_m}$ and $X|_{\G_m}\to X_1\times\G_m$ given at the level of points by $(a,\lambda)\mapsto\lambda\cdot a$ and $a\mapsto ((\pi(a))^{-1}\cdot a,\pi(a))$, respectively, are inverses to each other and are $\G_m$-equivariant. This proves the first part of the lemma, and gives us an isomorphism of graded rings $A/(t-1)\otimes\C[t,1/t]\simeq A[1/t]$. Taking the degree $0$ parts of both sides, the second assertion follows.
\end{proof}

\begin{lemma}
\label{2l2}
Let $A$ be a connected graded Cohen-Macaulay ring of pure dimension, and let $t\in A_1$ be non-nilpotent. Suppose that
\begin{enumerate}
	\item[\textup{(1)}]$X_1$ is a normal scheme, and
	\item[\textup{(2)}]$X_0$ is reduced of dimension $\dim A-1$.
\end{enumerate}
Then $A$ is a normal domain.
\end{lemma}

\begin{proof}
Since $A$ is connected, if it is a normal ring then it is automatically a normal domain. Since $A$ is Cohen-Macaulay, in order to show that $A$ is normal it suffices by Serre's criterion for normality \cite[Theorem 23.8, p.183]{matsumura} to show that $A$ is regular in codimension $1$.

Since $t$ is not nilpotent, the $\G_m$-equivariant morphism $\pi: X\to\A^1$ is dominant, with isomorphic fibers of dimension $\dim A-1$ over $\G_m$. By condition (2), in fact $\pi$ has equidimensional fibers everywhere. By the miracle flatness theorem \cite[Theorem 23.1, p.179]{matsumura}, since $A$ is Cohen-Macaulay and $\C[t]$ is regular it follows that $A$ is flat over $\C[t]$. In particular, $t$ is not a zero divisor in $A$.

By Lemma \ref{2l1}, we have $X|_{\G_m}\simeq X_1\times\G_m$, which is regular in codimension $1$ by condition (1). Thus, given any prime ideal $\pfrak\subset A$ of height $\leq 1$ such that $t\notin\pfrak$, the localization $A_\pfrak$ is regular. Next, suppose that $\pfrak\subset A$ is a prime ideal of height $1$ with $t\in\pfrak$. Since $t\in A$ is not a zero divisor, $A/(t)$ is Cohen-Macaulay of dimension $\dim A/(t)=\dim A-1$. Hypothesis (2) implies that $(A/tA)_\pfrak$ is regular, and hence $A_\pfrak$ is regular. This proves the result.
\end{proof}

\subsection{Rees algebras} \label{sect:2.2}
By a \emph{filtered algebra} $(A,\Fil)$ we mean an algebra $A$ and an increasing filtration $\Fil$ of $A$ by complex vector spaces, indexed by $\Z_{\geq0}$, such that:
\begin{enumerate}
	\item[\textup{(1)}] $\C\subseteq \Fil_0A$, and
	\item[\textup{(2)}] $\Fil_{r}\cdot\Fil_s\subseteq\Fil_{r+s}$ for every $r,s\geq0$.
\end{enumerate}
A morphism of filtered algebras $(B,\Fil)\to(A,\Fil)$ is a morphism of algebras $B\to A$ mapping $\Fil_r$ into $\Fil_r$ for every $r\geq0$. Given a $\Z_{\geq0}$-graded algebra $A$, we have an associated filtered algebra $(A,\Fil)$ where $\Fil$ is given by $\Fil_rA=\bigoplus_{s=0}^rA_s$.

\begin{example}
For the grading of $\C[t]$ by degree, or more generally the grading of a polynomial ring $\C[x_1,\dots,x_N]$ in several variables by (total) degree, we shall denote by $\Fil^{\deg}$ the associated filtration.
\end{example}

Given a filtered algebra $(A,\Fil)$, the Rees algebra $A^{\Fil}$ is the graded algebra
$$A^{\Fil}=\bigoplus_{r=0}^\infty \Fil_r.$$
A morphism of filtered algebras $(B,\Fil)\to(A,\Fil)$ induces a morphism $B^{\Fil}\to A^{\Fil}$ of Rees algebras. Given a filtered algebra $(A,\Fil)$, we shall denote by $t$ the element of degree $1$ in $A^{\Fil}$ corresponding to $1\in\Fil_1 A=A_1^{\Fil}$. We have identifications
$$A^{\Fil}/(t-1)=A\quad\text{and}\quad A^{\Fil}/(t)=\Gr^{\Fil}A$$
where the latter preserves gradings. Note that $\Spec A^{\Fil}$ is connected if and only if $\Spec A_0^{\Fil}$ is connected. We have a natural identification $A=A^{\Fil}[1/t]_0$ by Lemma \ref{2l1}. Thus, if $A^{\Fil}$ is a finitely generated graded algebra, then $\Proj A^{\Fil}$ is a projective scheme containing $\Spec A$ as an open affine complement to the closed subscheme $\Proj\Gr^{\Fil}A$. For simplicity of notation, given a filtered algebra $(A,\Fil^x)$ where $x$ is a symbol, we shall often write $A^x=A^{\Fil^x}$ for the Rees algebra.

\begin{example}
\label{rees}
We have the following.
\begin{enumerate}
	\item[(1)] Equipping $\C$ with the trivial filtration $\Fil$ given by $\Fil_r=\C$ for all $r\geq0$, we have $\C^{\Fil}=\C[t]$ with grading by degree. For any filtered algebra $(A,\Fil)$, the obvious morphism $(\C,\Fil)\to(A,\Fil)$ induces a $\G_m$-equivariant morphism of schemes $\pi:\Spec A^{\Fil}\to\A^1$.
	\item[(2)] Given a filtered algebra $(B,\Fil)$ and a surjective algebra homomorphism $\varphi:B\to A$, we have an induced filtration $\Fil$ on $A$ defined by $\Fil_r=\varphi(\Fil_r)$ for every $r\geq0$ making $\varphi$ a morphism of filtered algebras. We thus have an induced morphism $\varphi:B^{\Fil}\to A^{\Fil}$.
	\item[(3)] For $(A=\C[x_1,\dots,x_N],\Fil^{\deg})$ the polynomial ring with the filtration by degree, let $A^{\deg}$ be the Rees algebra. We then have an isomorphism $A^{\deg}\simeq\C[X_1,\dots,X_N,T]$ given by
	$$f(x_1,\dots,x_N)\mapsto T^r f(X_1/T,\dots,X_N/T)$$
for each $f\in\Fil_r^{\deg}A=A_r^{\deg}$. The projective space $\Proj A^{\deg}\simeq\P^N$ is the usual compactification of the affine space $\Spec A=\A^N$.
\end{enumerate}
\end{example}

\begin{lemma}
\label{lema}
Given a filtered algebra $(A,\Fil)$, we have the following.
\begin{enumerate}
	\item[\textup{(1)}] $A^{\Fil}$ is an integral domain if and only if $A$ is.
	\item[\textup{(2)}] If $A^{\Fil}$ is a normal domain, then so is $A$.
\end{enumerate}
\end{lemma}

\begin{proof}
(1) If $A^{\Fil}$ is an integral domain, then so is $A=A^{\Fil}[1/t]_0$. Conversely, suppose that $A$ is an integral domain. If $ab=0$ in $A^{\Fil}$, then replacing $a$ and $b$ by nonzero terms (if any) of highest degree we may assume that $a$ and $b$ are homogeneous of degree $r$ and $s$, respectively. We then have $ab=0\in A_{r+s}^{\Fil}=\Fil_{r+s}$ within $A$, and hence $a=0$ or $b=0$. This proves that $A^{\Fil}$ is an integral domain.

(2) Given a normal $\Z$-graded domain $B$, its degree $0$ component $B_0$ is also a normal domain. Indeed, it is obvious that $B_0$ is a domain; to see that $B_0$ is also normal, note that if $x\in B$ is integral over $B_0$ then we must have $x\in B_0$ by degree reasons. Applying this to $A=A^{\Fil}[1/t]_0$, we see that if $A^{\Fil}$ is a normal domain then so is $A$.
\end{proof}

\subsection{Character varieties of groups} \label{sect:2.3}

Let $\pi$ be a finitely generated group. The ($\SL_2$) \emph{representation variety} $\Rep_\pi$ is the affine scheme determined by the functor
$$A\mapsto\Hom(\pi,\SL_2(A))$$
for every commutative ring $A$. Given a sequence of generators of $\pi$ with $m$ elements, we have a presentation of $\Rep_\pi$ as a closed subscheme of $\SL_2^m$ defined by equations coming from relations among the generators. For each $a\in\pi$, let $\tr_a$ be the regular function on $\Rep_\pi$ given by $\rho\mapsto\tr\rho(a)$.

The ($\SL_2$) \emph{character variety} of $\pi$ over $\C$ is the affine invariant theoretic quotient
$$X_\pi=\Rep_\pi\git\SL_2=\Spec\C[\Rep_\pi]^{\SL_2(\C)}$$
under the simultaneous conjugation action of $\SL_2$. Note that the regular function $\tr_{a}$ for each $a\in\pi$ descends to a regular function on $X(\pi)$. Moreover, $X(\pi)$ has a natural model over $\Z$, defined as the spectrum of
$$R_\pi=\Z[\tr_a:a\in\pi]/(\tr_{1}-2,\tr_a\tr_b-\tr_{ab}-\tr_{ab^{-1}}).$$
The relations in the above presentation arise from the fact that the trace of the $2\times 2$ identity matrix is $2$, and $\tr(A)\tr(B)=\tr(AB)+\tr(AB^{-1})$ for every $A,B\in\SL_2(\C)$. We refer to \cite{horowitz}, \cite{ps}, \cite{saito} for details.

\begin{example}
\label{exfree}
We refer to Goldman \cite{goldman2} for details of examples below. Let $F_m$ denote the free group on $m\geq1$ generators $a_1,\dots,a_m$.
\begin{enumerate}
	\item[(1)] We have $\tr_{a_1}:X(F_1)\simeq\A^1$.
	\item[(2)] We have $(\tr_{a_1},\tr_{a_2},\tr_{a_1a_2}):X(F_2)\simeq\A^3$ by Fricke \cite[Section 2.2]{goldman2}.
	\item[(3)] The coordinate ring $\Q[X(F_3)]$ is the quotient of the polynomial ring
	$$\Q[\tr_{a_1},\tr_{a_2},\tr_{a_3},\tr_{a_1a_2},\tr_{a_2a_3},\tr_{a_1a_3},\tr_{a_1a_2a_3},\tr_{a_1a_3a_2}]$$
	by the ideal generated by two elements
$$
\tr_{a_1a_2a_3}+\tr_{a_1a_3a_2}-(\tr_{a_1a_2}\tr_{a_3}+\tr_{a_1a_3}\tr_{a_2}+\tr_{a_2a_3}\tr_{a_1}-\tr_{a_1}\tr_{a_2}\tr_{a_3})
$$
and
\begin{align*}
\tr_{a_1a_2a_3}\tr_{a_1a_3a_2}&-\{(\tr_{a_1}^2+\tr_{a_2}^2+\tr_{a_3}^2)+(\tr_{a_1a_2}^2+\tr_{a_2a_3}^2+\tr_{a_1a_3}^2)&\\
&\quad -(\tr_{a_1}\tr_{a_2}\tr_{a_1a_2}+\tr_{a_2}\tr_{a_3}\tr_{a_2a_3}+\tr_{a_1}\tr_{a_3}\tr_{a_1a_3})\\
&\quad +\tr_{a_1a_2}\tr_{a_2a_3}\tr_{a_1a_3}-4\}.
\end{align*}
\end{enumerate}
\end{example}

We record the following, which is attributed by Goldman \cite{goldman2} to Vogt \cite{vogt}.

\begin{lemma}
\label{rellem}
Given a finitely generated group $\pi$ and $a_1,a_2,a_3,a_4\in \pi$, we have
\begin{align*}
2{\tr_{a_1a_2a_3a_4}}&={\tr_{a_1}}{\tr_{a_2}}{\tr_{a_3}}{\tr_{a_4}}+{\tr_{a_1}}{\tr_{a_2a_3a_4}}+{\tr_{a_2}}{\tr_{a_3a_4a_1}}+{\tr_{a_3}}{\tr_{a_4a_1a_2}}\\
&\quad +{\tr_{a_4}}{\tr_{a_1a_2a_3}}+{\tr_{a_1a_2}}{\tr_{a_3a_4}}+{\tr_{a_4a_1}}{\tr_{a_2a_3}}-{\tr_{a_1a_3}}{\tr_{a_2a_4}}\\
&\quad -{\tr_{a_1}}{\tr_{a_2}}{\tr_{a_3a_4}}-{\tr_{a_3}}{\tr_{a_4}}{\tr_{a_1a_2}}-{\tr_{a_4}}{\tr_{a_1}}{\tr_{a_2a_3}}-{\tr_{a_2}}{\tr_{a_3}}{\tr_{a_4a_1}}.
\end{align*}
\end{lemma}

The above computation implies the following fact.

\begin{fact}
\label{fact}
If $\pi$ is a group generated by $a_1,\dots,a_m$, then $\Q[X(\pi)]$ is generated as a $\Q$-algebra by the collection $\{\tr_{a_{i_1}\dotsm a_{i_k}}:1\leq i_1<\dots<i_k\leq m\}_{1\leq k\leq 3}$.
\end{fact}

\subsection{Word compactifications} \label{sect:2.4}
Let $\pi$ be a finitely generated group. Recall that a \emph{length function} on $\pi$ is a function $\ell:\pi\to\R_{\geq0}$ satisfying the conditions:
\begin{enumerate}
	\item[\textup{(1)}] $\ell(1)=0$,
	\item[\textup{(2)}] $\ell(g^{-1})=\ell(g)$ for every $g\in \pi$, and
	\item[\textup{(3)}] $\ell(gh)\leq \ell(g)+\ell(h)$ for every $g,h\in \pi$.
\end{enumerate}
A length function equips $\pi$ with a pseudo-metric via the formula $d_\ell(g,h)=\ell(gh^{-1})$. Given a generating set $\sigma\subseteq \pi$, we have the \emph{$\sigma$-word length function} $\length_\sigma:\pi\to\Z_{\geq0}$ which is given by $\length_\sigma(a)=\min\{r:\text{$a=w_1\dotsm w_r$ for some $w_i\in \sigma\cup \sigma^{-1}$}\}$. Let $\Fil^\sigma$ be the increasing filtration of the coordinate ring $R_\pi$ of the character variety of $\pi$ given by
$$\Fil_r^\sigma R_\pi=\Span_\C\{{\tr_{a}}:a\in\pi,\length_\sigma(a)\leq r\}.$$
We have $\C=\C\cdot{\tr_{1}}=\Fil_0R_\pi$, and the relations ${\tr_{a}}{\tr_{a}}={\tr_{ab}}+{\tr_{ab^{-1}}}$ on $R_\pi$ show that $\Fil_r^\sigma\cdot\Fil_s^\sigma\subseteq\Fil_{r+s}^\sigma$ for every $r,s\in\Z_{\geq0}$. Thus, $(R_\pi,\Fil^\sigma)$ is a filtered algebra. Let $R_\pi^\sigma$ denote the associated Rees algebra.

\begin{definition}
The \emph{$\sigma$-word compactification} of $X_\pi$ is the scheme $X_\pi^\sigma=\Proj R_\pi^\sigma$.
\end{definition}

\begin{lemma}
\label{dimlem}
Let $\pi$ be a group generated by a finite set $\sigma$.
\begin{enumerate}
	\item[\textup{(1)}] We have $(R_\pi^\sigma)_0=\C$. For each $r\geq1$, we have $\dim_\C (R_\pi^\sigma)_r<\infty$.
	\item[\textup{(2)}] $R_\pi^\sigma$ is finitely generated by homogeneous elements of degree $\leq3$.
\end{enumerate}
In particular, $X_\pi^\sigma$ is projective and can be presented as a closed subscheme of some weighted projective space $\P(w_1,\dots,w_s)$ with $w_i\in\{1,2,3\}$.
\end{lemma}

\begin{proof} (1) We have $\length_\sigma(g)=1$ if and only if $g=1$, and there are at most finitely many $\sigma$-words of length $\leq r$ in $\pi$ for any given $r$.

(2) Applying Lemma \ref{rellem} and induction, we see that the class $\tr_{w}\in R_\pi$ of every word $w$ of length $r\geq4$ in $\pi$ is a linear combination of terms of the form $\tr_{w_1}\dotsm\tr_{w_s}$ with $\sum_{i=1}^s\length_\sigma(w_i)\leq r$ and $\length_\sigma(w_i)\leq3$ for each $i=1,\dots,s$. This shows that $R_\pi^\sigma$ is generated by $(R_\pi^\sigma)_1\oplus(R_\pi^\sigma)_2\oplus (R_\pi^\sigma)_3$. Applying part (1) of this lemma, we obtain the desired result.
\end{proof}

\begin{remark}
When $\pi$ is a free group with free generating set $\sigma$ (which is the case mainly considered in this paper), the filtration $\Fil^\sigma$ on $R_\pi$ defined above gives a special case of the filtrations, constructed by Manon \cite{manon}, on the free group character variety coming from valuations. More precisely, for the unique graph with $|\sigma|$ edges and one vertex together with an assignment of length 1 to each edge, the general construction in \cite{manon} applies to give a valuation on $R_{\pi}$ whose associated filtration agrees with $\Fil^\sigma$ (see in particular Lemma 8.3 \emph{loc.cit.}).
\end{remark}

\subsection{An alternate description} \label{sect:2.5}
Following the notation from the beginning of this paper, let $\M$ be the scheme parametrizing $2\times 2$ matrices. Using the isomorphism $(x_{ij}):\M\simeq\A^4$ via matrix entries, we equip $\C[\M^m]$ for each $m\geq1$ with grading by degree and associated filtration $\Fil^{\deg}$. We have $\SL_2$ acting on $\M^m$ by simultaneous conjugation, preserving the grading (and hence the filtration) on $\C[\M^m]$. Therefore, $\SL_2$ acts on the Rees algebra $\C[\M^m]^{\deg}$ of $(\C[\M^m],\Fil^{\deg})$ preserving the grading. By Example \ref{rees}.(3), we have an identification
$$\C[\M^m]^{\deg}=\C[(X_1)_{ij},\dots,(X_m)_{ij},T]$$
with the polynomial ring on $4m+1$ generators, so that $\Spec\C[\M^m]^{\deg}=\M^m\times\A^1$. The $\SL_2$-action is the product of the conjugation action on $\M^m$ and the trivial action on $\A^1$. The morphism $\pi:\M^m\times\A^1\to\A^1$ from Example \ref{rees}.(1) is just the projection onto the second factor.

Let $\pi$ be a group with a finite generating sequence $\sigma$ of length $m$. The choice of $\sigma$ gives us a $\SL_2$-equivariant closed immersion $\sigma^*:\Rep_\pi\hookrightarrow\Rep_m=\SL_2^m\subset\M^m$, where $\Rep_m=\Rep_{F_m}$ is the representation variety of the free group $F_m$ on $m$ generators. Via the surjective morphism $\C[\M^m]\to\C[\Rep_\pi]$, the filtration $\Fil^{\deg}$ on $\C[\M^m]$ induces a filtration on $\C[\Rep_\pi]$ which we shall denote $\Fil^{\sigma}$.

\begin{proposition}
\label{propmorph}
We have $\Fil_r^\sigma R_\pi=(\Fil_r^\sigma\C[\Rep_\pi])^{\SL_2}$ for every $r\geq0$. In particular, there is a surjective morphism of graded rings
$$\Phi_\sigma:(\C[\M^m]^{\deg})^{\SL_2}\to R_\pi^\sigma.$$
\end{proposition}

\begin{proof}
We begin by considering the filtration $\Fil^{\deg}$ on $\C[\M^m]$. For each $r\geq0$, a close inspection of the work of Procesi \cite{procesi} (Theorems 1.1, 1.2, and 1.3 \emph{loc.cit.}) shows that the space of invariants $(\C[\M^m]_r)^{\SL_2}$ is spanned by elements of the form
$$\tr(w_1)\tr(w_2)\dotsm \tr(w_s)$$
where each $w_i$ is a noncommutative monomial in the standard matrix variables $x_1,\dots,x_m$ on $\M^m$ such that we have $\sum_{i=1}^s \deg(w_i)=r$. Here, the degree $\deg(w_i)$ refers to the total multiplicity of matrices $x_1,\dots,x_m$ appearing in the monomial $w_i$. Since $\SL_2$ is linearly reductive, we see that the map
$$(\Fil_r^{\deg}\C[\M^m])^{\SL_2}\to(\Fil_r^{\sigma}\C[\Rep_\pi])^{\SL_2},$$
obtained from the surjective morphism
$$\Fil_r^{\deg}\C[\M^m]\to \Fil_r^{\sigma}\C[\Rep_\pi]$$
of finite-dimensional complex representations of $\SL_2$, remains surjective. In particular, the vector space $(\Fil_r^\sigma\C[\Rep_\pi])^{\SL_2}$ is spanned by elements of the form
$$\tr_{w_1}\tr_{w_2}\dotsm\tr_{w_s}$$
where each $w_i$ is a product of elements of $\sigma$ (with multiplicities, but without inverses) such that the sum total number of elements appearing is $\leq r$. This shows that $(\Fil_r^\sigma\C[\Rep_\pi])^{\SL_2}\subseteq \Fil_r^{\sigma}R_\pi$ since each element of the above form lies in $\Fil_r^{\sigma}R_\pi$, recalling that the filtration $\Fil^\sigma$ is compatible with multiplication on $R_\pi$. The other containment $\Fil_r^\sigma R_\pi\subseteq(\Fil_r^{\sigma}\C[\Rep_\pi])^{\SL_2}$ follows from the observation that, using the relation ${\tr_{a}}{\tr_{a}}={\tr_{ab}}+{\tr_{ab^{-1}}}$ in $R_\pi$, one can write any ${\tr_{a}}$ with $\length_\sigma(a)\leq r$ as a linear combination of elements of the above form. Finally, since the surjection $\C[\M^m]\to\C[\Rep_\pi]$ is compatible with the filtrations and the $\SL_2$-actions, we obtain a surjection
\begin{align*}
(\C[\M^m]^{\deg})^{\SL_2}&\to(\C[\Rep_\pi]^\sigma)^{\SL_2}=(\C[\Rep_\pi]^{\SL_2})^\sigma=R_\pi^\sigma
\end{align*}
which proves the second part of the proposition.
\end{proof}

\section{Log Calabi-Yau property}
\label{sect:3}
In this section, we give a proof of the log Calabi-Yau property of the moduli spaces $X_{g,n,k}$ introduced in Section \ref{sect:1}. In Section \ref{sect:3.1}, we analyze the homogeneous coordinate ring of the word compactifications of character varieties for free groups, relying on the invariant theoretic results of Hochster-Roberts \cite{hr} and Le Bruyn \cite{lebruyn} as well as a Hilbert series characterization of graded Gorenstein domains among graded Cohen-Macaulay domains due to Stanley \cite{stanley}. In Section \ref{sect:3.2}, we combine the results from Section \ref{sect:3.1} with the analysis of singularities in Section \ref{sect:4} to deduce Theorem \ref{mtheorem}, using the works of Demazure \cite{demazure} and Watanabe \cite{watanabe}.

\subsection{Free groups} \label{sect:3.1}
Let $F_m$ be the free group on $m\geq1$ generators $\sigma$. As before, let $\Rep_m=\SL_2^m$ be the representation variety of $F_m$, and let $X_m=\Rep_m\git\SL_2$ be the character variety of $F_m$ with $R_m=\C[\Rep_m]^{\SL_2}$. Let $\Fil^\sigma$ be the filtration on the coordinate ring $\C[\Rep_m]$ of $\Rep_m=\SL_2^m$ defined in Section \ref{sect:2.5}. In other words, it is the filtration obtained by the projection of $\Fil^{\deg}$ under the surjective ring homomorphism $\C[\M^m]\to\C[\Rep_m]$. For $i\in\{1,\dots,m\}$, let us write
$$E_i=\det(X_i)-T^2$$
for the homogeneous element in $\C[\M^m]^{\deg}$ of degree $2$, with the identification $\C[\M^m]^{\deg}=\C[(X_1)_{ij},\dots,(X_m)_{ij},T]$ from Section \ref{sect:2.5}.

\begin{lemma}
\label{fl}
We have the following.
\begin{enumerate}
	\item[\textup{(1)}] $\C[\M^m]^{\deg}/(E_1,\dots,E_m)$ is a normal graded Cohen-Macaulay domain.
	\item[\textup{(2)}] $E_1,\dots,E_m$ is a regular sequence in $\C[\M^m]^{\deg}$.
	\item[\textup{(3)}] We have $\C[\Rep_m]^\sigma\simeq\C[\M^m]^{\deg}/(E_1,\dots,E_m)$.
\end{enumerate}
\end{lemma}

\begin{proof}
Let us write $\Rep_m'=\Spec\C[\M^m]^{\deg}/(E_1,\dots,E_m)$. Our proof of (1) and (2) will proceed by induction on $m$. Both statements are clear when $m=1$ (we note by the Jacobian criterion that the singular locus of $\Rep_1'$ consists of a single point). So let $m\geq2$. For any $s\in\{1,\dots,m\}$, note that we have
$$\C[\M^m]^{\deg}/(E_1,\dots,E_s)=\C[\Rep_s']\otimes\C[(X_{s+1})_{ij},\dots,(X_m)_{ij}].$$
In particular, each
$$\C[\M^m]^{\deg}/(E_1,\dots,E_s)$$
is an integral domain, showing that $E_1,\dots,E_m$ is a regular sequence in $\C[\M^m]^{\deg}$, and $\C[\Rep_m']$ is a graded Cohen-Macaulay ring of pure dimension $3m+1$. Consider the $\G_m$-equivariant morphism
$$\pi:\Rep_m'\to\A^1$$
associated to the morphism of graded rings $\C[t]\to \C[\Rep_m']$ given by $t\mapsto T$. The fiber $\pi^{-1}(1)=\SL_2^m$ is smooth, and $\pi^{-1}(0)=\M_0^m$ is reduced of dimension $3m$. We thus conclude by Lemma \ref{2l2} that $\C[\Rep_m']$ is a normal domain. This completes the induction, and we have proven (1) and (2). It remains to prove (3). Note that we have a natural surjective ring homomorphism $\C[\Rep_m']\to\C[\Rep_m]^\sigma$. Note that $\C[\Rep_m]^\sigma$ is integral by Lemma \ref{lema}, and has dimension at least $3m+1$. Thus, we see that $\C[\Rep_m']\to\C[\Rep_m]^\sigma$ must be an isomorphism, as the former is also integral of dimension $3m+1$. This completes the proof of the lemma.
\end{proof}

Let $\Fil^\sigma$ be the $\sigma$-word filtration on $R_m$ defined in Section \ref{sect:3.1}, and let $R_m^\sigma$ be the associated Rees algebra.

\begin{proposition}
\label{propcm}
The ring $R_m^\sigma$ is a normal graded Cohen-Macaulay domain.
\end{proposition}

\begin{proof}
First, by a result of Hochster-Roberts \cite[Main Theorem]{hr}, we see that $(\C[\M^m]^{\deg})^{\SL_2}$ is Cohen-Macaulay. We next claim that the surjective homomorphism $\Phi_\sigma:(\C[\M^m]^{\deg})^{\SL_2}\to R_m^\sigma$ constructed in Proposition \ref{propmorph} induces an isomorphism
$$(\C[\M^m]^{\deg})^{\SL_2}/(E_1,\dots,E_m)\simeq R_m^\sigma.$$
Indeed, since $(\C[\M^m]^{\deg})^{\SL_2}$ is a pure subring of $\C[\M^m]^{\deg}$, the left hand side is isomorphic to $(\C[\M^m]^{\deg}/(E_1,\dots,E_m))^{\SL_2}=(\C[\Rep_m]^\sigma)^{\SL_2}=R_m^\sigma$ by Lemma \ref{fl}, which is the desired result. Since $\C[\Rep_m]^\sigma$ is a normal graded domain by Lemma \ref{fl}, so is $R_m^\sigma$. To show that $R_m^\sigma$ is Cohen-Macaulay, it suffices to show that $E_1,\dots,E_m$ is a regular sequence in $(\C[\M^m]^{\deg})^{\SL_2}$. But this follows from the fact that $E_1,\dots,E_m$ is a regular sequence in $\C[\M^m]^{\deg}$ and that $(\C[\M^m]^{\deg})^{\SL_2}$ is a pure subring of $\C[\M^m]^{\deg}$. The desired result follows.
\end{proof}

The Hilbert series $H_m(t)$ of the graded algebra $R_m^\sigma$ is by definition the following formal power series. Note that $H_m(t)$ is well-defined by Lemma \ref{dimlem}.
$$H_m(t)=\sum_{r=0}^\infty(\dim_\C (R_m^\sigma)_r)t^r=\sum_{r=0}^\infty (\dim_\C\Fil_r^\sigma R_m)t^r\in\Z[[t]].$$

\begin{theorem}
\label{hilbert}
Assume that $m\geq2$. Then $H_m(t)$ is rational and satisfies
$$H_m(1/t)=(-1)^{3m-2}t^{2m+1}H_m(t).$$
In particular, $R_m^\sigma$ is Gorenstein with canonical module $R_m^\sigma(-2m-1)$.
\end{theorem}

\begin{proof}
For $m=2$ and $F_2$ a free group on generators $\sigma=\{a_1,a_2\}$, we may verify the functional equation by direct computation as follows. Recall the isomorphism $(\tr_{a_1},\tr_{a_2},\tr_{a_1a_2}):X_2\simeq\A^3$ from Example \ref{exfree}. It follows $R_2^\sigma\simeq\C[X_1,X_2,X_3,T]$ with each $X_1,X_2,T$ given degree 1 and $X_3$ given degree $2$, so that the Krull dimension of $R_2^\sigma$ is $4$ and
$$H_2(t)=\frac{1}{(1-t)^3(1-t^2)}$$
from which the desired symmetry follows. For $m\geq3$, we have the following result of Le Bruyn \cite{lebruyn} (based on rational expressions due to Weyl and Schur, cf.~\emph{loc.cit.}): the Hilbert series $h_m(t)$ of the graded algebra $\C[\M^m]^{\SL_2}$ has the functional equation $h_m(1/t)=-t^{4m}h_m(t)$. Our Hilbert series $H_m(t)$ is related to the series $h_m(t)$ by
$$H_m(t)=\frac{(1-t^2)^m}{1-t}h_m(t),$$
where the factor $(1-t^2)^m/(1-t)$ arises as we take the $\SL_2$-invariants of
$$\C[\M^m]^{\deg}/(E_1,\dots,E_m)$$
instead of $\C[\M^m]$, noting that $E_1,\dots,E_m$ is a regular sequence of homogeneous elements of degree $2$ in $\C[\M^m]^{\deg}$. Thus, the functional equation of $H_m(t)$ follows from that of $h_m(t)$. Now, the Krull dimension of $R_m^\sigma$ is $(3m-3)+1$. Hence, our last claim follows from the functional equation and Proposition \ref{propcm}, by the result of Stanley \cite[Theorem 4.4]{stanley} (see also \cite[Corollary 4.4.6, p.177]{brh}).
\end{proof}

\subsection{Relative character varieties} \label{sect:3.2}
Let $\Sigma_{g,n}$ be a compact oriented surface of genus $g\geq0$ with $n\geq1$ boundary components such that $\chi(\Sigma_{g,n})=2-2g-n<0$. We fix a standard presentation of the fundamental group
$$\pi_1(\Sigma_{g,n})=\langle a_1,\dots,a_{2g+n}|[a_1,a_2]\dotsm[a_{2g-1},a_{2g}] a_{2g+1}\dotsm a_{2g+n}\rangle.$$
By our assumption, $\pi_1(\Sigma_{g,n})$ is a free group of rank $m=2g+n-1\geq2$. We shall refer to the set $\sigma=\{a_1,\dots,a_{2g+n-1}\}$ of free generators as \emph{standard}.

Let $\Rep_{g,n}=\Rep_{\pi_1(\Sigma_{g,n})}$ be the representation variety of $\Sigma_{g,n}$ (or of $\pi_1(\Sigma_{g,n})$), and let $X_{g,n}=\Rep_{g,n}\git\SL_2$ be the character variety of $\Sigma_{g,n}$. The latter is also the coarse moduli space of $\SL_2(\C)$-local systems on $\Sigma_{g,n}$. Note that the standard free generators $\sigma$ of the fundamental group gives us isomorphisms
\begin{align*}
&\sigma^*:\Rep_{g,n}\simeq\Rep_m=\SL_2^m,\\
&\sigma^*:X_{g,n}\simeq X_m
\end{align*}
where $\Rep_m$ and $X_m$ are as in Section \ref{sect:3.1}.
For each $k=(k_1,\dots,k_n)\in\A^n(\C)$, the corresponding fibers $\Rep_{g,n,k}$ and $X_{g,n,k}$ of the morphisms
\begin{align*}
&(\tr(x_{2g+1}),\dots,\tr(x_{2g+n-1}),\tr(x_{2g+n})):\Rep_{g,n}\to\A^n,\\
&(\tr_{a_{2g+1}},\dots,\tr_{a_{2g+n}}):X_{g,n}\to\A^n
\end{align*}
(where $x_1,\dots,x_{2g+n-1}$ are the matrix variables of $\SL_2^m$ and we abbreviated here $x_{2g+n}=(\langle x_1,x_2\rangle\dotsm\langle x_{2g-1},x_{2g}\rangle x_{2g+1}\dotsm x_{2g+n-1})^*$ for simplicity) are called the \emph{relative representation variety} and \emph{relative character variety} of $\Sigma_{g,n}$. Let us denote $R_{g,n}=\C[X_{g,n}]$ and by $R_{g,n,k}=\C[X_{g,n,k}]$.

\begin{example}
\label{exrel}
We review the presentations of some relative character varieties for surfaces of small Euler characteristic, following Goldman \cite{goldman2}.
\begin{enumerate}
	\item[\textup{(1)}] $(g,n)=(0,3)$. By Example \ref{exfree}.(2), we have $({\tr_{a_1}},{\tr_{a_2}},{\tr_{a_1a_2}}):X_{0,3}\simeq \A^3$. Since the $3$ boundary components of $\Sigma_{0,3}$ correspond to $a_1$, $a_2$, and $a_1a_2$, we see that $X_{0,3,k}$ is simply a point for every $k\in\A^3(\C)$.
	\item[\textup{(2)}] $(g,n)=(1,1)$. By Example \ref{exfree}.(2), we have
	$$({\tr_{a_1}},{\tr_{a_2}},{\tr_{a_1a_2}}):X_{1,1}\simeq\A_{(x,y,z)}^3$$
	where $(x,y,z)$ denotes the sequence of standard coordinate functions on $\A^3$. The boundary component of $\Sigma_{1,1}$ corresponds to $\langle a_1,a_2\rangle$, which defines the function given by the expression
	$$[\langle a_1,a_2\rangle]=f(x,y,z)=x^2+y^2+z^2-xyz-2$$
	obtained using the identities ${\tr_{a}}{\tr_{b}}={\tr_{ab}}+{\tr_{ab^{-1}}}$ in the coordinate ring $R_2$. Thus, $X_{1,1,k}$ for $k\in\A^1(\C)$ is an affine cubic surface which is given by a level set $f^{-1}(k)$ in $\A^3$ of the above function, also called a (generalized) Markoff surface. Looking at the partial derivatives of the above function, its critical locus is given by the conditions
	\begin{align*}
	2x=yz,\quad 2y=xz,\quad \text{and}\quad 2z=xy.
	\end{align*}
	The critical locus is thus $\{(0,0,0),(s_12,s_22,s_32):s_1s_2s_3=1, s_i\in\{\pm1\}\}$, and every relative character variety $X_{1,1,k}$ is normal.
	\item[\textup{(3)}] $(g,n)=(0,4)$. Using Example \ref{exfree}.(3), given $k=(k_1,\dots,k_4)\in\A^4(\C)$ we have a closed immersion
		$$({\tr_{a_1a_2}},{\tr_{a_2a_3}},{\tr_{a_1a_3}}):X_{0,4}\hookrightarrow\A_{(x,y,z)}^3$$
		with the image given by the affine cubic surface
	$$f_k(x,y,z)=x^2+y^2+z^2+xyz-ax-by-cz-d=0$$
	where we define $a=k_1k_2+k_3k_4$, $b=k_2k_3+k_1k_4$, $c=k_1k_3+k_2k_4$, and $d=4-k_1^2-k_2^2-k_3^2-k_4^2-k_1k_2k_3k_4$. The critical locus of the function $f$ is given by the conditions
	$$2x+yz-a=0,\quad 2y+xz-b=0,\quad 2z+xy-c=0.$$
	We see that the locus $\subset\A^3$ defined by the above three equations is finite. Hence, every relative character variety $X_{0,4,k}$ is normal.
\end{enumerate}
\end{example}

From Section \ref{sect:2}, we have the $\sigma$-word filtration $\Fil^\sigma$ on the coordinate ring $R_{g,n}$ with Rees algebra $R_{g,n}^\sigma=(\C[\Rep_{g,n}]^\sigma)^{\SL_2}$. For fixed $k=(k_1,\dots,k_n)\in\A^n(\C)$ as before, let $F_1,\dots,F_{n-1}$ and $F_n$ respectively be the homogeneous elements in $R_{g,n}^\sigma$ of degree $1,\dots,1$ and $4g+n-1$ given by
\begin{align*}
F_1&={\tr_{a_{2g+1}}}-k_1t,\quad\dots,\quad F_{n-1}={\tr_{a_{2g+n-1}}}-k_{n-1}t,\quad\text{and}\\
F_n&={\tr_{a_{2g+n}}}-k_nt^{4g+n-1}.
\end{align*}
Here, $F_n$ is homogeneous since $\length_\sigma(a_{2g+n})=4g+n-1$. Let us define the quotient
$$R_{g,n,k}^\sigma=R_{g,n}/(F_1,\dots,F_n).$$
(This notation is unambiguous; it will follow from our work that the above ring is the same as the Rees algebra associated to the projection of the $\sigma$-word filtration $\Fil^\sigma$ on $R_{g,n,k}$.) Note that $R_{g,n,k}^\sigma/(t-1)=R_{g,n,k}$. The pair
$$(Z,D)=(\Proj R_{g,n,k}^\sigma,\Proj R_{g,n,k}^\sigma/(t))$$
is our projective compactification and boundary divisor of $X_{g,n,k}$ mentioned in Theorem \ref{mtheorem}. The goal of this section is to verify that $(Z,D)$ satisfies the desired properties.

\begin{theorem}
\label{gorenst}
The ring $R_{g,n,k}^\sigma$ is a normal graded Gorenstein domain of dimension $6g+2n-5$, with canonical module $R_{g,n,k}^\sigma(-1)$.
\end{theorem}

\begin{proof}
We have an identification $R_{g,n}^\sigma=R_m^\sigma$. We claim that $F_1,\dots,F_n$ is a regular sequence in $R_m^\sigma=(\C[\Rep_m]^\sigma)^{\SL_2}$. Indeed, using the fact that $(\C[\M^m]^{\deg})^{\SL_2}$ is a pure subring of $\C[\M^m]^{\deg}$, it suffices to note that $\C[\Rep_m]^\sigma/(F_1,\dots,F_s)$ is an integral domain (of successively lower dimension) for $1\leq s\leq n-1$. This follows from an easy application of Lemma \ref{2l2} (as in the proof of Lemma \ref{fl}). Now, by Theorem \ref{hilbert}, $R_m^\sigma$ is a graded Gorenstein ring of dimension $3m-2$ with canonical module $R_m^\sigma(-2m-1)$.
Hence, it follows that
$$R_{g,n,k}^\sigma=R_{g,n}^\sigma/(F_1,\dots,F_n)$$
is a graded Gorenstein ring of dimension $3m-2-n=6g+2n-5$ with canonical module
$$R_{g,n,k}^\sigma(-2m-1+(n-1+4g+n-1))=R_{g,n,k}^\sigma(-1),$$
provided that we know $R_{g,n,k}^\sigma$ is a domain.
Indeed, $R_{g,n,k}^\sigma$ is a graded Cohen-Macaulay ring, being a quotient of $R_{g,n}^\sigma$ by a regular sequence $F_1,\dots,F_n$, and its Hilbert series is $(1-t)^{n-1}(1-t^{4g+n-1})$ times the Hilbert series of $R_{g,n}^\sigma$; by applying parts (c) and (b) of \cite[Corollary 4.4.6, p.177]{brh} and \cite[Theorem 4.4.5, p.176]{brh}, we would then deduce the above.

It thus remains to show that $R_{g,n,k}^\sigma$ is a normal domain. By Lemma \ref{2l2}, it suffices to show the following:
\begin{enumerate}
	\item[\textup{(1)}] $R_{g,n,k}$ is a normal ring, and
	\item[\textup{(2)}] $R_{g,n,k}^\sigma/(t)$ is reduced of dimension $6g+2n-6$.
\end{enumerate}
If $\chi(\Sigma_{g,n})\leq-3$, then (1) and (2) will follow from Propositions \ref{6s1} and \ref{6s3} proved in Section \ref{sect:4}. Let us now consider the cases where $|\chi(\Sigma_{g,n})|\leq2$. If $\chi(\Sigma_{g,n})=-2$, then (1) and (2) follow from Propositions \ref{6s2} and \ref{6s3} for $(g,n)=(1,2)$, and they follow from Example \ref{exrel}.(3) and Proposition \ref{6s3} for $(g,n)=(0,4)$. The case $\chi(\Sigma_{g,n})=-1$ remains. The case $(g,n)=(0,3)$ is trivial, and for $(g,n)=(1,1)$ Example \ref{exrel}.(2) verifies (1). It thus remains to verify (2) for $(g,n)=(1,1)$. But we find that
$$R_{1,1,k}^\sigma/(t)\simeq\C[x,y,z]/(z(z-xy)),$$
whose spectrum is the union of the plane $\{z=0\}$ with the surface $\{z=xy\}$ in $\A_{(x,y,z)}^3$, and hence is reduced and of dimension $2$. This completes the proof.
\end{proof}

Let $Z$ be the compactification of $X_{g,n,k}$ as above. Note that $Z$ is a normal irreducible projective scheme, since $R_{g,n,k}^\sigma$ is a normal domain by Theorem \ref{gorenst}. By the result of Demazure \cite[Th\'eor\`eme (3.5) p.51]{demazure}, there is an ample $\Q$-divisor $E$ on $Z$ such that
$$R_{g,n,k}^\sigma=\bigoplus_{r=0}^\infty H^0(Z,\Ocal_Z(rE)) t^r$$
where $t\in(R_{g,n,k}^\sigma)_1$ is the image of $t\in(R_{g,n}^\sigma)_1$ under the projection $R_{g,n}^\sigma\to R_{g,n,k}^\sigma$. Following Demazure (\emph{loc.cit.}), the $\Q$-divisor $E$ is given explicitly as follows. Let us write $\Div(t)=\sum p_F F$ for the Weil divisor on $\Spec R_{g,n,k}^\sigma$ determined by $t$. Note that each $F$ is given by $\Spec B_F$ where $B_F$ is a graded integral quotient of $R_{g,n,k}^\sigma$. Defining $q_F\geq0$ by
$$(B_F)_n\neq0\iff n\in q_F\Z,$$
the $\Q$-divisor $E$ on $Z$ is given by
$$E=\sum_{q_F\neq0}\frac{p_F}{q_F}\Proj(B_F).$$
\begin{lemma}
\label{deglem}
We have $p_F\in\{0,1\}$ and $q_F\in\{0,1,2\}$ for every prime Weil divisor $F$ of $\Spec R_{g,n,k}^\sigma $.
\end{lemma}

\begin{proof}
The claim $p_F\in\{0,1\}$ follows from the fact that
$$\Spec(R_{g,n,k}^\sigma/t)=(\C[\Rep_{g,n,k}']/(t))^{\SL_2}$$
is reduced by Proposition \ref{6s3}. To see that $q_F\in\{0,1,2\}$, note that $B_F$ is generated in degrees $1$, $2$, and $3$ by Lemma \ref{dimlem} so we must have $q_F\in\{0,1,2,3\}$. Suppose toward contradiction that $q_F=3$, so that $(B_F)_1=(B_F)_2=0$. Given any $a_1,a_2,a_3\in\pi_1(\Sigma_{g,n})$, on $R_{g,n,k}$ we have the relations
$$
\tr_{a_1a_2a_3}+\tr_{a_1a_3a_2}=\tr_{a_1a_2}\tr_{a_3}+\tr_{a_1a_3}\tr_{a_2}+\tr_{a_2a_3}\tr_{a_1}-\tr_{a_1}\tr_{a_2}\tr_{a_3}
$$
and
\begin{align*}
\tr_{a_1a_2a_3}\tr_{a_1a_3a_2}&=(\tr_{a_1}^2+\tr_{a_2}^2+\tr_{a_3}^2)+(\tr_{a_1a_2}^2+\tr_{a_2a_3}^2+\tr_{a_1a_3}^2)&\\
&\quad -(\tr_{a_1}\tr_{a_2}\tr_{a_1a_2}+\tr_{a_2}\tr_{a_3}\tr_{a_2a_3}+\tr_{a_1}\tr_{a_3}\tr_{a_1a_3})\\
&\quad +\tr_{a_1a_2}\tr_{a_2a_3}\tr_{a_1a_3}-4.
\end{align*}
Since the right hand sides of both equations are zero on $B_F$ which is an integral domain, this shows that $\tr_{a_1a_2a_3}=\tr_{a_1a_3a_2}=0$ on $B_F$. This shows that we must have $(B_F)_3=0$ as well, a contradiction. Hence, we must have $q_F\in\{0,1,2\}$.
\end{proof}

Let $D=Z\setminus X_{g,n,k}$ be the reduced boundary divisor on $Z$. As a Weil divisor, $D$ is a formal sum $\sum_V V$ of prime divisors $V$ where $V=\Proj B_F$ runs over integral quotients $B_F$ of $R_{g,n,k}^\sigma/(t)$ by minimal homogeneous prime ideals. We are ready to prove the main part of Theorem \ref{mtheorem}, restated as follows.

\begin{theorem}
We have $K_Z+D\sim0$.
\end{theorem}

\begin{proof}
Combining Corollary \ref{gorenst} with the result of Watanabe \cite[Corollary 2.9]{watanabe}, we find that $K+E'+E=\Div(f)$ for some $f\in\C(Z)$ (the function field of $Z$) and
$$E'=\sum_{\substack{q_F\neq0}}\frac{q_F'-1}{q'_F}\Proj B_F,$$
where we have set $q_F'=q_F/\gcd(p_F,q_F)$ for $q_F\neq0$, so that in particular $q_F'=1$ if $p_F=0$.
But in light of Lemma \ref{deglem}, we have
\begin{align*}
0&\sim K+E'+E\\
&=K+\sum_{\substack{p_F=1\\q_F=2}}\frac{1}{2}\Proj B_F+\left\{\sum_{\substack{p_F=1\\q_F=2}}\frac{1}{2}\Proj B_F+\sum_{\substack{p_F=1\\q_F=1}}\Proj B_F\right\}\\
&=K+D
\end{align*}
from which we conclude the result.
\end{proof}

\begin{example}\label{examples}
For the values of $(g,n)$ considered below, one can describe the compactification $Z$ of $X_{g,n,k}$ and compute the Hilbert series $H_D(t)$ of $D$ (or of the graded ring $R_{g,n,k}^\sigma/(t)$) directly.

\begin{enumerate}
	\item $(g,n)=(1,1)$. By Example \ref{exfree}.(2), the word compactification $X_{1,1}^\sigma$ of $X_{1,1}$ is the weighted projective three-space $\P(1,1,1,2)$ and our compactification $Z$ is a hypersurface of degree $4$. As mentioned in the example of Section \ref{sect:1.3}, we have
	$$H_D(t)=\frac{1+t^2}{(1-t)^2}.$$
	\item $(g,n)=(0,4)$. The fundamental group of $\Sigma$ is free of rank $3$. By Example \ref{exfree}.(3), the word compactification $X_{0,4}^\sigma$ of $X_{0,4}$ is a complete intersection of two hypersurfaces of degree $3$ and $6$ in the weighted projective space
	$$\P(1,1,1,1,2,2,2,3,3).$$
	Here, the weights above appear from the presentation of $R_{0,4}$ as a quotient of the polynomial ring $\Q[\tr_{a_1},\tr_{a_2},\tr_{a_3},\tr_{a_1a_2},\tr_{a_2a_3},\tr_{a_1a_3},\tr_{a_1a_2a_3},\tr_{a_1a_3a_2}]$. Our compactification $Z$ of $X_{0,4,k}$ is then a complete intersection of four hypersurfaces of degree $1$, $1$, $1$, and $3=4g+n-1$ in $X_{0,4}^\sigma$. We therefore have
	$$H_D(t)=\frac{(1-t^3)(1-t^6)}{(1-t)^3(1-t^2)^3(1-t^3)^2}\cdot (1-t)^3(1-t^3)=\frac{1-t^6}{(1-t^2)^3}.$$
	\item $(g,n)=(1,2)$. As in the previous case, the fundamental group of $\Sigma$ is free of rank $3$. By Example \ref{exfree}.(3), the word compactification $X_{1,2}^\sigma$ is a complete intersection of two hypersurfaces of degree $3$ and $6$ in the weighted projective space $\P(1,1,1,1,2,2,2,3,3)$.
Our compactification $Z$ of $X_{1,2,k}$ is then a complete intersection of two hypersurfaces of degrees $1$ and $5=4g+n-1$ in $X_{1,2}^\sigma$. We therefore have
	$$H_D(t)=\frac{(1-t^3)(1-t^6)}{(1-t)^3(1-t^2)^3(1-t^3)^2}\cdot (1-t)(1-t^5)=\frac{(1+t^3)(1-t^5)}{(1-t)^2(1-t^2)^3}.$$
\end{enumerate}
In each of the above cases, we can verify directly that $H_D(1/t)=H_D(t)$, a symmetry which holds for general $(g,n)$ by Theorem \ref{gorenst} and \cite[Corollary 4.4.6, p.177]{brh}. Our work in Section \ref{sect:5} will show that $H_D(t)=Z_{g,n}(t)$ for general $(g,n)$. Thus, even when an explicit presentation for $X_{g,n,k}$ is not given, the series $H_D(t)=Z_{g,n}(t)$ can be computed using combinatorial arguments. For instance, the combinatorics in Section \ref{sect:6} can be used to show that
$$Z_{2,1}(t)=\frac{(1-t^8)(1+t^2+4t^3+t^4+t^6)}{(1-t)^4(1-t^2)^5}$$
(equal to $(1-t^8)Z_4(t)$ in the notation of Section \ref{sect:6}); we omit the details.
\end{example}

\section{Analysis of singularities} \label{sect:4}

The purpose of this section is to provide an analysis of singularities for various algebras appearing in Section \ref{sect:3}. For the convenience of the reader, in this paragraph we provide a brief outline of this section. Using the notation of the proof of Theorem \ref{gorenst}, what need to be proved are:
\begin{enumerate}
	\item[\textup{(1)}] $R_{g,n,k}$ is a normal ring, and
	\item[\textup{(2)}] $R_{g,n,k}^\sigma/(t)$ is reduced of dimension $6g+2n-6$.
\end{enumerate}
The goal of this section is to establish this (except in certain special cases which are dealt with separately in the proof of Theorem \ref{gorenst}) at the level of representation varieties. That is, we shall show in Section \ref{sect:4.3} that
\begin{enumerate}
	\item[\textup{(1)}] the coordinate ring of $\Rep_{g,n,k}$ from Section \ref{sect:3.2} is normal (Propositions \ref{6s1} and \ref{6s2}), and
	\item[\textup{(2)}] the coordinate ring $H_{g,n}$, which is defined below and whose ring of $\SL_2(\C)$-invariants recovers $R_{g,n,k}^\sigma/(t)$, is reduced (Proposition \ref{6s3}).
\end{enumerate} (The fact that $H_{g,n}$ has dimension $6g+2n-3$ (whence $R_{g,n,k}^\sigma/(t)$ has dimension $6g+2n-6$) is easily seen by arguing as in the beginning of the proof of Theorem \ref{gorenst}.) The advantage of working with representation varieties and $H_{g,n}$ is that their singular loci can be partly understood using conditions on matrix variables, which are obtained in Section \ref{sect:4.1} in the form of Lemma \ref{crit}. The main work in Section \ref{sect:4.3} is therefore to bound from above the dimensions of spaces of matrices cut out by such conditions. To this end, various preliminary dimension bounds and computations are carried out in Section \ref{sect:4.2} and the Appendix.

Let $m\geq1$ be an integer. Recall from Section \ref{sect:2.5} that we have an identification
$$\C[\M^m]^{\deg}=\C[(X_1)_{ij},\dots,(X_m)_{ij},T]=\C[\M^m\times\A^1].$$
As before, for $i=1,\dots,m$, let $E_i=\det(X_i)-T^2$ be the element of degree $2$ in $\C[\M^m]^{\deg}$. Recall that $\C[\Rep_m]^\sigma=\C[\M^m]^{\deg}/(E_1,\dots,E_m)$, where $\Rep_m$ is the representation variety of the free group on $m$ generators $\sigma$.

Let $g\geq0$ and $n\geq1$ be integers with $2g+n-1=m$. Fix $k=(k_1,\dots,k_n)\in\A^n(\C)$. Let $F_1,\dots,F_{n-1}$ and $F_n$ respectively be the homogeneous elements in $\C[\M^m]^{\deg}$ of degrees $1,\dots,1$ and $4g+n-1$ given by
\begin{align*}
F_1&=\tr(X_{2g+1})-k_1T,\quad\dots,\quad F_{n-1}=\tr(X_{2g+n-1})-k_{n-1}T,\quad\text{and}\\
F_n&=\tr(\langle X_1,X_2\rangle\dotsm\langle X_{2g-1},X_{2g}\rangle X_{2g+1}\dotsm X_{2g+n-1})-k_nT^{4g+n-1}.
\end{align*}
Let us define $\Rep_m'=\Spec\C[\Rep_m]^\sigma$ and $\Rep_{g,n,k}'=\Spec\C[\Rep_m']/(F_1,\dots,F_n)$. The singular locus of $\Rep_{g,n,k}'$ is the intersection of $\Rep_{g,n,k}'$ with the locus of critical points $\Crit(E,F)$ of the morphism $(E_1,\dots,E_m,F_1,\dots,F_n):\M^m\times\A^1\to\A^{m+n}$. Consider now the $\G_m$-equivariant morphism
$$\pi:\Rep_{g,n,k}'\to\A^1$$
of affine schemes induced by the the morphism of graded rings $\C[t]\to\C[\Rep_{g,n,k}']$ sending $t\mapsto T$. Note that $\pi^{-1}(1)\simeq\Rep_{g,n,k}$ is the relative representation variety of the compact surface $\Sigma_{g,n}$ of genus $g$ with $n$ boundary components corresponding to $k$, as introduced in Section \ref{sect:3.2}. Let us define $H_{g,n}=\pi^{-1}(0)$.

\subsection{Critical points} \label{sect:4.1}
Fix integers $g\geq0$ and $n\geq1$ such that $m=2g+n-1\geq1$. Fix a complex number $t\in\C$. For each $s\in\C$ such that $s^2=t$, the fiber of the morphism (considered in the proof of Lemma \ref{fl})
$$\pi:\Rep_m'\to\A^1$$
above $s$ is the scheme $\M_t^m$. By definition, a point $a\in\M_t^m(\C)$ is represented by an $m$-tuple of matrices $a=(a_1,\dots,a_m)$, each $a_i$ having determinant $t$. Given such a point, let us introduce the following supplementary notation.
\begin{enumerate}
	\item[\textup{(1)}] Let $b_i=\langle a_{2i-1},a_{2i}\rangle$ for each $i\in\{1,\dots,g\}$.
	\item[\textup{(2)}] Let $b_{g+j}=a_{2g+j}$ for each $j\in\{1,\dots,n-1\}$.
	\item[\textup{(3)}] In generality, for a cyclic interval $I=(i_0,i_0+1,\dots,i_0+k)$ in $\Cfrak_{g+n-1}$ (cf.~Section \ref{sect:1.5}), let $b_I=b_{i_0}\dotsm b_{i_0+k}$ where the indices are considered modulo $g+n-1$. Also, let $b^I=b_{I^c}$ when $I^c$ is defined. Let us write $b^i=b^{(i)}$ for simplicity.
\end{enumerate}
For example, we have
$$b_{(1,\dots,g+n-1)}=b_1\dotsm b_{g+n-1}=\langle a_1,a_2\rangle\dotsm \langle a_{2g-1},a_{2g}\rangle a_{2g+1}\dotsm a_{2g+n-1}.$$
We shall describe, in terms of these parameters, the critical locus of the morphism $F=(F_1,\dots,F_n):\Rep_m'\to\A^n$ defined in Section \ref{sect:4.1} restricted to $\M_t^m\subset\Rep_m'$. We begin with a lemma on tangent spaces.

\begin{lemma}
\label{tanlem}
Let $a\in\M(\C)$ be given.
\begin{enumerate}
	\item[\textup{(1)}] $T_{a}\M_{\det a}=\{v\in\M(\C):\tr(va^*)=0\}$. We have $\mathfrak{sl}_2(\C)\cdot a\subseteq T_{a}\M_{\det a}$ and $a\cdot\mathfrak{sl}_2(\C)\subseteq T_a\M_{\det a}$. The containments are equalities if $\det(a)\neq0$.
	\item[\textup{(2)}] $T_{a}\M_{\det a,\tr a}=\{v\in\M(\C):\tr(v)=\tr(va^*)=0\}$. We have a containment $[\mathfrak{sl}_2(\C),a]\subseteq T_{a}\M_{\det a,\tr a}$ which is an equality if $a$ is nonscalar.
\end{enumerate}
\end{lemma}
\begin{proof} (1) For any $a,v\in\M(\C)$, we have $\det(a+\varepsilon v)=\det(a)+\varepsilon\tr(va^*)$ within $\C[\varepsilon]/(\varepsilon^2)$, from which the first statement follows. In particular, $\dim T_{a}\M_{\det a}=3$ provided that $a\neq0$. Since $\tr(uaa^*)=\tr(aua^*)=\det(a)\tr(u)=0$ for every $u\in\mathfrak{sl}_2(\C)$, we have $\mathfrak{sl}_2(\C)\cdot a\subseteq T_{a}\M_{\det a}$ and $a\cdot\mathfrak{sl}_2(\C)\subseteq T_a\M_{\det a}$. Lastly, if $\det(a)\neq0$ then $\dim\mathfrak{sl}_2(\C)\cdot a=\dim a\cdot\mathfrak{sl}_2(\C)=3$ from which the equalities follow by dimension reasons.

(2) For any $a,v\in\M(\C)$ we have $\tr(a+\varepsilon v)=\tr(a)+\varepsilon\tr(v)$ within $\C[\varepsilon]/(\varepsilon^2)$, from which the first statement follows. We have $\tr([v,a]a^*)=\tr(vaa^*-ava^*)=\det(a)\tr(v)-\det(a)\tr(v)=0$ and $\tr([v,a])=0$ for any $v\in\mathfrak{sl}_2(\C)$, and hence $[\mathfrak{sl}_2(\C),a]\subseteq T_{a,{\Char}}$. Finally, if $a$ is nonscalar, then we have $\dim_\C T_{a}\M_{\det a,\tr a}=2$ since $v\mapsto \tr(va^*)$ and $w\mapsto \tr(v)$ are linearly independent on $\M(\C)$. On the other hand, $\dim_\C[\mathfrak{sl}_2(\C),a]=2$ if $a$ is nonscalar, and hence $T_{a}\M_{\det a,\tr a}=[\mathfrak{sl}_2(\C),a]$ by dimension reasons.
\end{proof}

\begin{lemma}
\label{crit}
A critical point $(a_1,\dots,a_m)$ of $F:\M_t^m\to\A^n$ must satisfy at least one of the following two conditions.
\begin{enumerate}
	\item[\textup{(1)}] $a_{2g+j}$ is scalar for some $j\in\{1,\dots,n-1\}$.
	\item[\textup{(2)}] $[b_i,b^i]=0$ for every $i\in\{1,\dots,g+n-1\}$, and
\begin{align*}
&[a_{2i-1},a_{2i}a_{2i-1}^*a_{2i}^*b^i]=[a_{2i},a_{2i-1}^*a_{2i}^*b^ia_{2i-1}]\\
&\quad =[a_{2i-1}^*,a_{2i}^*b^ia_{2i-1}a_{2i}]=[a_{2i}^*,b^ia_{2i-1}a_{2i}a_{2i-1}^*]=0
\end{align*}
for every $i\in\{1,\dots,g\}$.
\end{enumerate}
\end{lemma}

\begin{proof}
Assume that $a\in\M_t^m(\C)$ satisfies neither of the conditions (1) and (2). We must show that $d_aF:T_a\M_t^m\to\C^n$ is surjective. Now, the first $n-1$ coordinates of $d_aF$ are given by
$$(u_1,\dots,u_m)\mapsto(\tr(u_{2g+1}),\dots,\tr(u_{2g+n-1}))$$
for each $u=(u_1,\dots,u_m)\in T_a\M_{t}^m$.
Since $a_{2g+j}$ is nonscalar for $j\in\{1,\dots,n-1\}$ by hypothesis, each of the functionals $u\mapsto\tr(u_{2g+j})$ is nonzero. Thus, to show surjectivity of $d_aF$ it remains to show that $(0,\dots,0,1)\in\Img d_aF$.

Let $u\in\mathfrak{sl}_2(\C)$ be given. We then have $[u,b_{g+j}]\in T_{b_{g+j}}\M_t$ for $j\in\{1,\dots,n-1\}$, and the composition $T_{b_{g+j}}\M_t\hookrightarrow T_a\M_t^m\to\C^n$ with $d_aF$ sends
$$[u,b_{g+j}]\mapsto(0,\dots,0,\tr([u,b_{g+j}]b^{g+j}))=(0,\dots,0,\tr(u[b_{g+j},b^{g+j}])).$$
Similarly, after a short computation (using the fact $u^*=-u$ for $u\in\mathfrak{sl}_2(\C)$) we also find that for $i\in\{1,\dots,g\}$ the composition $T_{(a_{2i-1},a_{2i})}\M_t^2\hookrightarrow T_a\M_t^m\to\C^n$ with $d_aF$ sends
\begin{align*}
(ua_{2i-1}-a_{2i-1}u,-a_{2i}u)&\mapsto(0,\dots,0,\tr(u[a_{2i-1},a_{2i}a_{2i-1}^*a_{2i}^*b^i])),\\
(a_{2i-1}u,0)&\mapsto(0,\dots,0,\tr(u[a_{2i},a_{2i-1}^*a_{2i}^*b^ia_{2i-1}])),\\
(0,a_{2i}u)&\mapsto(0,\dots,0,\tr(u[a_{2i-1}^*,a_{2i}^*b^ia_{2i-1}a_{2i}])),\\
(-a_{2i-1}u,ua_{2i}-a_{2i}u)&\mapsto(0,\dots,0,\tr(u[a_{2i}^*,b^ia_{2i-1}a_{2i}a_{2i-1}^*])).
\end{align*}
In particular, summing up the above tangent vectors in $T_a\M_t^m$ we find that
$$(ua_{2i-1}-a_{2i-1}u,ua_{2i}-a_{2i}u)\mapsto(0,\dots,0,\tr(u[b_i,b^i])).$$
Thus, if $a$ does not satisfy condition (2), at least one of the expressions above must be nonzero for some $u\in\mathfrak{sl}_2(\C)$. This implies that $(0,\dots,0,1)\in\Img d_aF$, and hence $d_aF$ is surjective, as desired.
\end{proof}

\subsection{Dimension estimates} \label{sect:4.2}
We collect a number of estimates for the dimensions of certain schemes parametrizing sequences of matrices with prescribed conditions. We will later use them to analyze the singularities of $\Rep_{g,n,k}$ and $H_{g,n}$. Throughout Section \ref{sect:4.3}, fix integers $g\geq0$ and $n\geq0$ such that $m=2g+n-1\geq1$.

Fix $k=(k_1,\dots,k_n)\in\A^n(\C)$, and let $\Phi:\SL_2^{2g}\times\prod_{i=1}^n\SL_{2,k_i}\to\SL_2$ be the morphism given by
$$(a_1,\dots,a_{2g+n})\mapsto\langle a_1,a_2\rangle\dotsm\langle a_{2g-1},a_{2g}\rangle a_{2g+1}\dotsm a_{2g+n}.$$

\begin{lemma}
\label{6d1}
$\Phi$ is flat above $\SL_2\setminus\{\pm\mathbf1\}$ with fibers of dimension $6g+2n-3$.
\end{lemma}

\begin{proof}
First, $\Phi$ is dominant by Lemma \ref{a3}. Note that $\SL_2^{2g}\times\prod_{i=1}^n\SL_{2,k_i}$ is a complete intersection scheme of dimension $6g+2n$, and in particular Cohen-Macaulay. The target $\SL_2$ of $\Phi$ is regular of dimension $3$. By the miracle flatness theorem, to prove that $\Phi$ is flat it suffices to show that the fibers of $\Phi$ are equidimensional, of dimension $6g+2n-3$. Consider the composition
$\tr\circ\Phi:\SL_2^{2g}\times\prod_{i=1}^n\SL_{2,k_i}\to\A^1$. Since $\tr\circ\Phi$ is dominant, it is flat. In particular, the fibers of $\tr\circ\Phi$ are all of pure dimension $6g+2n-1$.

Next, note that $\SL_2$ acts by conjugation on the domain and target of $\Phi$, and $\Phi$ is equivariant with respect to this action. Hence, given any $l\in\C$ the fibers of the restriction $\Phi:\Phi^{-1}(\SL_{2,l})\to\SL_{2,l}$, away from the fibers over $\pm\mathbf1$, are isomorphic hence equidimensional of dimension $6g+2n-3$.
\end{proof}

\begin{lemma}
\label{6d2}
If $2g+n-1\geq2$, then $\Phi^{-1}(\pm\mathbf1)$ has dimension at most $6g+2n-3$.
\end{lemma}

\begin{proof}
Let us denote by $F_+$ (reps.~$F_-$) the fiber of $\Phi$ above $\mathbf1$ (resp.~$-\mathbf1$). We shall proceed by induction on $(g,n)$. Consider first the case where $n=0$. We must then have $g\geq2$. Consider the projection
$$\pi:F_{\pm}\subset\SL_2^{2g}\to\SL_2^2$$
onto the last two copies of $\SL_2$. Given $(b_1,b_2)\in\SL_2^2(\C)$, the dimension of the fiber $\pi^{-1}(b_1,b_2)$ is estimated as follows.
\begin{enumerate}
	\item[\textup{(1)}] If $\langle b_1,b_2\rangle\neq\pm\mathbf1$, then $\dim\pi^{-1}(b_1,b_2)=6(g-1)-3$ by Lemma \ref{6d1}.
	\item[\textup{(2)}] If $\langle b_1,b_2\rangle=\pm\mathbf1$, then $\dim\pi^{-1}(b_1,b_2)\leq 6(g-1)-2$.
\end{enumerate}
Here, part (2) follows from the inductive hypothesis and the fact that, when $g=2$ so that $g-1=1$, the fiber of $\langle-,-\rangle:\SL_2^2\to\SL_2$ above $\pm\mathbf1$ has dimension at most $4$ by Lemma \ref{a4}. Since the locus of $(b_1,b_2)\in\SL_2^2$ defined by part (1) is $6$-dimensional, and the locus defined by condition (2) is at most $4$-dimensional by Lemma \ref{a4}, we see that $\dim F_\pm\leq 6g-3$ as desired. Consider next the case where $n\geq1$. If $g=0$, then we must have $n\geq3$. Consider the projection
$$\pi:F_{\pm}\subset\SL_2^{2g}\times\prod_{i=1}^n\SL_{2,k_i}\to\SL_{2,k_n}$$
onto the last factor of the product. Given $b\in\SL_{2,k_n}(\C)$, the dimension of the fiber $\pi^{-1}(b)$ is estimated as follows.
\begin{enumerate}
	\item[\textup{(1)}] If $b\neq\pm\mathbf1$, then $\dim\pi^{-1}(b)=6g+2(n-1)-3$ by the Lemma \ref{6d1}.
	\item[\textup{(2)}] If $b=\pm\mathbf1$, then $\dim\pi^{-1}(b)\leq 6g+2(n-1)- 2$.
\end{enumerate}
Here, part (2) follows from the inductive hypothesis, from Lemma \ref{a4} in the case $(g,n)=(1,1)$, and from the observation that the fiber of $\SL_{2,l}\times\SL_{2,\pm l}\to\SL_2$, $(a_1,a_2)\mapsto a_1a_2$ over $\pm\mathbf1$ has dimension $2$ in the case $(g,n)=(0,3)$. Since the locus of $b$ satisfying (1) has dimension $2$, we find that $\dim F_\pm\leq 6g+2n-3$ as desired.
\end{proof}

Let $\Psi:\M_0^{2g}\times\M_{0,0}^n\to\M_0$ be the morphism given by
$$(a_1,\dots,a_{2g+n})\mapsto\langle a_1,a_2\rangle\dotsm\langle a_{2g-1},a_{2g}\rangle a_{2g+1}\dotsm a_{2g+n}.$$

\begin{lemma}
\label{6d3}
$\Psi$ is flat above $\M_0\setminus\{0\}$ with fibers of dimension $6g+2n-3$.
\end{lemma}

\begin{proof}
First, $\Psi$ is dominant by Lemma \ref{a7}. Note that $\M_{0}^{2g}\times\M_{0,0}^n$ is a complete intersection scheme of dimension $6g+2n$, and in particular Cohen-Macaulay. The target $\M_0$ of $\Psi$ is regular away from $\{0\}$, of dimension $3$. By the miracle flatness theorem, to prove that $\Psi$ is flat over $\M_0\setminus\{0\}$ it suffices to show that the fibers of $\Psi$ are equidimensional, of dimension $6g+2n-3$. Consider the composition
$\tr\circ\Psi_{g,n}:\M_{0}^{2g}\times\M_{0,0}^n\to\A^1$. Since $\tr\circ\Psi$ is dominant, it is flat. In particular, the fibers of $\tr\circ\Psi$ are all of pure dimension $6g+2n-1$.

Next, note that $\SL_2$ acts by conjugation on the domain and target of $\Psi$, and $\Psi$ is equivariant with respect to this action. Hence, given any $l\in\C$ the fibers of the restriction $\Psi:\Psi^{-1}(\M_{0,l})\to\M_{0,l}$ are isomorphic, hence equidimensional of dimension $6g+2n-3$, away from the fiber over $0$.
\end{proof}

\begin{lemma}
\label{6d4}
$\Psi^{-1}(0)$ has dimension at most $6g+2n-1$.
\end{lemma}

\begin{proof}
Since the domain $\M_0^{2g}\times\M_{0,0}^n$ of $\Psi$ has dimension $6g+2n$ and is integral, if $\dim\Psi^{-1}(0)\geq 6g+2n$ then $\Psi$ must be identically zero, contradicting the fact that $\Psi$ is dominant by Lemma \ref{a7}.
\end{proof}

\subsection{Singularities} \label{sect:4.3}
Let $\Sigma_{g,n}$ be a compact oriented surface of genus $g\geq0$ with $n\geq1$ boundary components and $\chi(\Sigma_{g,n})=2-2g-n<0$. The fundamental group of $\Sigma_{g,n}$ is free of rank $m=2g+n-1\geq2$. Let $k=(k_1,\dots,k_n)\in\A^n(\C)$, and let $F_1,\dots,F_n\in\C[\M^m]^{\deg}$ be as in Section \ref{sect:4.1}.

\begin{proposition}
\label{6s1}
If $m\geq4$, the scheme $\Rep_{g,n,k}$ is normal.
\end{proposition}

\begin{proof}
Since $\Rep_{g,n,k}$ is a complete intersection of dimension $6g+2n-3$ in the regular scheme $\M^{m}$, it is Cohen-Macaulay. As a consequence of Serre's criterion for normality \cite[Theorem 23.8, p.183]{matsumura}, a Cohen-Macaulay ring is normal if and only if it is regular in codimension $1$. It thus suffices to demonstrate that $\Rep_{g,n,k}$ is regular in codimension $1$. Consider the restriction of $F=(F_1,\dots,F_n)$ above to $\SL_2^m\subset\M^m\times\A^1$. Since $\SL_2^m$ is regular, the singular locus of $\Rep_{g,n,k}=F^{-1}(0)$ is the intersection of $\Rep_{g,n,k}$ with the critical locus of $F$. It suffices to show that the locus in $\Rep_{g,n,k}$ defined by each of the conditions $(1)$ and $(2)$ of Lemma \ref{crit} has dimension at most $(6g+2n-3)-2$. We shall represent points of $\Rep_{g,n,k}$ by sequences $(a_1,\dots,a_{2g+n-1})$ of matrices, and adopt the same notations $b_I$ and $b^I$ for the various products of matrices as used in Lemma \ref{crit}. 

First, consider the locus $W\subset\Rep_{g,n,k}$ defined by condition (1) of Lemma \ref{crit}, which plays a role only when $n\geq2$. For each $i=1,\dots,n-1$ and $s\in\{\pm1\}$, the locus $\{b_{g+i}=s\mathbf1\}\subset\Rep_{g,n,k}$ is isomorphic to $\Rep_{g,n-1,k'}$ where $k'$ is the $(n-1)$-tuple obtained from $k$ by omitting $k_i$ and replacing $k_n$ by $sk_n$. By our assumption on $(g,n)$, we have $2g+(n-1)-1\geq2$, and we find by Lemma \ref{6d2}
$$\dim\{b_{g+i}=s\mathbf1\}=\dim\Rep_{g,n-1,k'}=6g+2(n-1)-3=(6g+2n-3)-2.$$
As $W=\bigcup_{i=1}^{n-1}\{b_{g+i}=\pm\mathbf1\}$, it has codimension $2$ in $\Rep_{g,n,k}$ as desired.

Next, let $Z\subset\Rep_{g,n,k}$ be the locus defined by condition (2) of Lemma \ref{crit}. We stratify $Z$ further into three subloci $Z_1$, $Z_2$, and $Z_3$, and estimate their dimensions as follows.
\begin{enumerate}
	\item[\textup{(1)}] Let $Z_1\subset Z$ be the sublocus consisting of $(a_1,\dots,a_m)$ such that $b_i=\pm\mathbf1$ for some $i\in\{1,\dots,g\}$ or $b_{g+j}b_{g+j+1}=\pm\mathbf1$ for some $j\in\{1,\dots,n-2\}$. We claim that $\dim Z_1\leq (6g+2n-3)-2$.
	
	Given $i\in\{1,\dots,g\}$ and $s\in\{\pm1\}$, consider the locus where $b_i=s\mathbf1$. By Lemma \ref{a4}, $(a_{2i-1},a_{2i})$ must vary over a locus of dimension at most $4$. For fixed $(a_{2i-1},a_{2i})$ the remaining matrices in the sequence $(a_1,\dots,a_m)$ vary over a locus isomorphic to $\Rep_{g-1,n,k'}$ for $k=(k_1,\dots,k_{n-1},sk_n)$, which has dimension $\dim\Rep_{g-1,n,k'}=6(g-1)+2n-3$ by Lemma \ref{6d2} since $2(g-1)+n-1\geq2$ by our assumption. Thus, the locus in $Z_1$ where where $b_i=s\mathbf1$ for $i\in\{1,\dots,g\}$ has dimension bounded by
	$$4+(6(g-1)+2n-3)=(6g+2n-3)-2.$$

Given $j\in\{1,\dots,n-2\}$ and $s\in\{\pm1\}$, consider the locus where $b_{g+j}b_{g+j+1}=s\mathbf1$. The pair $(b_{g+j},b_{g+j+1})$ then varies over a locus of dimension at most $2$. Tor fixed $(b_{g+j},b_{g+j+1})$ the remaining matrices in the sequence $(a_1,\dots,a_m)$ vary over a locus isomorphic to $\Rep_{g,n-2,k'}$ where $k'$ is obtained from $k$ by removing $k_{j}$ and $k_{j+1}$, and replacing $k_n$ by $sk_n$. We have $\dim\Rep_{g,n-2,k'}=6g+2(n-2)-3$ by Lemma \ref{6d2} since $2g+(n-2)-1\geq2$ by our assumption. Thus, the given locus has dimension bounded by
	$$2+(6g+2(n-2)-3)=(6g+2n-3)-2.$$
	
	This shows that $\dim Z_1\leq (6g+2n-3)-2$.

	\item[\textup{(2)}] Let $Z_2\subset Z\setminus Z_1$ be the sublocus consisting of $(a_1,\dots,a_m)$ with $\tr(b_i)=\pm2$ for some $i\in\{1,\dots,g\}$ or $\tr(b_{g+j}b_{g+j+1})=\pm2$ for some $j\in\{1,\dots,n-2\}$. We claim that $\dim Z_2\leq(6g+2n-3)-2$.
	
	Given $i\in\{1,\dots,g\}$, consider the locus $\tr(b_i)=\pm2$. By Lemma \ref{6d1}, the pair $(a_{2i-1},a_{2i})$ varies over a locus of dimension at most $5$. For fixed $(a_{2i-1},a_{2i})$, by Lemma \ref{a2} the product $b^i$ must vary over a locus of dimension at most $1$. We have two possibilities.
	\begin{enumerate}
		\item We have $b^i\in\{\pm\mathbf1\}$. We must have $2(g-1)+(n-1)-1\geq2$ since $Z_2$ lies in the complement of $Z_1$. By Lemma \ref{6d2}, for fixed $(a_{2i-1},a_{2i})$ the remaining matrices in $(a_1,\dots,a_m)$ vary over a locus of dimension at most $6(g-1)+2(n-1)-3$.
		\item We have $b^i\neq\pm\mathbf1$. By Lemma \ref{6d1}, for fixed $(a_{2i-1},a_{2i})$ and $b^i$ the remaining matrices in $(a_1,\dots,a_m)$ vary over a locus of dimension at most $6(g-1)-2(n-1)-3$.
	\end{enumerate}
	Thus, we see that the locus of $(a_1,\dots,a_m)$ in $Z$ satisfying $\tr(b_i)=\pm2$ for some $i\in\{1,\dots,g\}$ has dimension bounded by
	$$5+1+(6(g-1)+2(n-1)-3)=(6g+2n-3)-2.$$
	
	Given $j\in\{1,\dots,n-2\}$, consider the locus $\tr(b_{g+j}b_{g+j+1})=\pm2$. By Lemma \ref{6d1}, we see that the pair $(b_{g+j},b_{g+j+1})$ varies over a locus of dimension $3$. For fixed $(b_{g+j},b_{g+j+1})$, by Lemma \ref{a2} the product $b^{\{g+j,g+j+1\}}$ must vary over a locus of dimension at most $1$. We have two possibilities.
	\begin{enumerate}
		\item We have $b^{\{g+j,g+j+1\}}\in\{\pm\mathbf1\}$. We must have $2g+(n-3)-1\geq2$ since $Z_2$ lies in the complement of $Z_1$. By Lemma \ref{6d2}, for fixed $(b_{g+j},b_{g+j+1})$ the remaining matrices in $(a_1,\dots,a_m)$ vary over a locus of dimension at most $6g+2(n-3)-3$.
		\item We have $b^{\{g+j,g+j+1\}}\neq\pm\mathbf1$. By Lemma \ref{6d2}, for fixed $(b_{g+j},b_{g+j+1})$ and $b^{\{g+j,g+j+1\}}$ the remaining matrices in $(a_1,\dots,a_m)$ vary over a locus of dimension at most $6g-2(n-3)-3$.
	\end{enumerate}
	Thus, we see that the locus of $(a_1,\dots,a_m)$ in $Z$ with $\tr(b_{g+j}b_{g+j+1})=\pm2$ for some $j\in\{1,\dots,n-2\}$ has dimension bounded by
	$$3+1+(6g+2(n-3)-3)=(6g+2n-3)-2.$$
	
	This shows that $\dim Z_2\leq (6g+2n-3)-2$.
	\item[\textup{(3)}] We claim that $Z_3=Z\setminus(Z_1\cup Z_2)$ must have $\dim Z_3\leq(6g+2n-3)-2$. Suppose first that $g\geq1$. The pair $(a_1,a_2)$ varies over a locus of dimension $6$. For fixed $(a_1,a_2)$, since $\tr(b_1)\neq\pm2$ we see by Lemma \ref{a2} that there are only finitely many possible values of $b^i$. For fixed $(a_1,a_2)$ and value of $b^i$, the remaining matrices in $(a_1,\dots,a_m)$ vary over a locus of dimension $6(g-1)+2(n-1)-3$ by Lemma \ref{6d1} (noting that $2(g-1)+(n-1)-1\geq2$ if $b^i\in\{\pm\mathbf1\}$ by assumption that $Z_3\cap Z_1=\emptyset$). Hence, for $g\geq1$, $Z_3$ has dimension bounded by
 	$$6+(6(g-1)+2(n-1)-3)=(6g+2n-3)-2.$$
	Arguing similarly, if $g=0$ so that $n\geq5$, $Z_3$ has dimension bounded by
	$$4+(6g+2(n-3)-3)=(6g+2n-3)-2.$$
	Therefore, $\dim Z_3\leq (6g+2n-3)-2$.
\end{enumerate}
From the above computations, we conclude that $\dim Z\leq (6g+2n-3)-2$. Hence, $\Rep_{g,n,k}$ is normal, as desired.
\end{proof}

\begin{proposition}
\label{6s2}
$\Rep_{1,2,k}$ is normal.
\end{proposition}

\begin{proof}
Since $\Rep_{1,2,k}$ is a complete intersection of dimension $7$ in the regular scheme $\M^m$, it is in particular Cohen-Macaulay. By Serre's criterion, it suffices to show that $\Rep_{1,2,k}$ is regular in dimension $1$. It suffices to show that locus defined by each of conditions (1) and (2) of Proposition \ref{crit} is of dimension at most $5$. We shall represent a point of $\Rep_{1,2,k}$ by a triple of matrices $(a_1,a_2,a_3)\in\SL_2^3$ satisfying $\tr(a_3)=k_1$ and $\tr(\langle a_1,a_2\rangle a_3)=k_2$, and adopt the notations of Section \ref{sect:4.1}.

Consider first the locus of $\Rep_{1,2,k}$ where $a_i=\pm\mathbf1$ for some $i\in\{1,2,3\}$. By Lemmas \ref{6d1} and \ref{a4}, we see that this locus has dimension at most $3+2-1=4$ when $i\in\{1,2\}$, and at most $3+3-1=5$ when $i=3$. The union $Y=\bigcup_{i=1}^3\{a_i=\pm\mathbf1\}$ includes the locus of $\Rep_{1,2,k}$ defined by condition (1) of Proposition \ref{crit}, and is at most $5$-dimensional.

It remains to estimate the dimension of the locus $Z\subset\Rep_{1,2,k}\setminus Y$ defined by condition (2) of Proposition \ref{crit}. In what follows, we stratify $Z$ into subloci and estimate their respective dimensions.
\begin{enumerate}
	\item[\textup{(1)}] Let $Z_1\subset Z$ be the locus where $\tr(a_1)\neq\pm2$. From the condition
	$$[a_1^{-1},a_2^{-1}a_3a_1a_2]=0,$$
	and Lemma \ref{a6}, we see that for given $a_1$ there are only finitely many possible values of $a_2^{-1}a_3a_1a_2$. The sublocus where $a_3a_1=\pm\mathbf1$ has dimension at most $2+3=5$, so we may restrict our attention to the sublocus of $Z_1$ where $a_3a_1\neq\pm\mathbf1$. For fixed $a_1$, there are only finitely many possible values of $\tr(a_3a_1)$, and hence $a_3$ varies over a locus of dimension $1$, and for fixed $a_1$ and $a_3$ since there are only finitely many possible values of $a_2^{-1}a_3a_1a_2$ showing that $a_2$ varies over a locus of dimension $1$ (the centralizer of $a_1a_3\neq\pm\mathbf1$ being $1$-dimensional). Thus, the sublocus of $Z_1$ where $a_3a_1\neq\pm\mathbf1$ has dimension at most $3+1+1=5$. Therefore, we conclude that $\dim Z_1\leq 5$.
	\item[\textup{(2)}] Let $Z_2\subset Z\setminus Z_1$ be the locus where $\tr(a_2)\neq\pm2$. We then consider
	$$[a_2,a_1^{-1}a_2^{-1}a_3a_1]=0,$$
	and by the same argument as in part (1) we conclude that $\dim Z_2\leq 5$.
	\item[\textup{(3)}] Let $Z_3\subset Z\setminus(Z_1\cup Z_2)$ be the locus where $\tr(a_1a_2)\neq\pm2$. By the condition
	$$[a_1a_2,a_1^{-1}a_2^{-1}a_3]=0$$
	and Lemma \ref{a6}, for fixed $a_1$ and $a_2$ there are at most finitely many possible values of $a_1^{-1}a_2^{-1}a_3$ and hence of $a_3$. Thus, $\dim Z_4\leq 2+2=4$.
	\item[\textup{(4)}] Let $Z_4=Z\setminus(Z_1\cup Z_2\cup Z_3)$. We must have $\tr(a_1),\tr(a_2),\tr(a_1a_2)\in\{\pm2\}$, and hence $(a_1,a_2)$ varies over a locus of dimension at most $3$. Hence, $\dim Z_4\leq 3+2=5$.
\end{enumerate}
This completes the proof that $\dim Z\leq 5$, and hence $\Rep_{1,2,k}$ is normal.
\end{proof}

\begin{proposition}
\label{6s3}
If $m\geq3$, the intersection $\Crit(E,F)\cap H_{g,n}$ has codimension at least 1 in $H_{g,n}$, and in particular $H_{g,n}$ is reduced.
\end{proposition}

\begin{proof}
First, the locus in $H_{g,n}$ determined by the condition $a_i=0$ has dimension
\begin{align*}
&3+6(g-1)+2(n-1)=(6g+2n-3)-2&&\text{if $i\in\{1,\dots,2g\}$},\\
&6g+2(n-2)=(6g+2n-3)-1&&\text{if $i\in\{2g+1,\dots,2g+n-1\}$}.
\end{align*}
Therefore, the union $Y=\bigcup_{i=1}^m\{a_i=0\}$ has dimension at most $(6g+2n-3)-1$. Since $\Crit(E)\cap H_{g,n}$ and the locus defined by condition (1) of Lemma \ref{crit} both lie in $Y$, their dimensions also cannot be more than $(6g+2n-3)-1$.

Next, let $Z\subset H_{g,n}\setminus Y$ be the locus defined by condition (2) of Lemma \ref{crit}. We must show that $\dim Z\leq(6g+2n-3)-1$. Let us first treat the cases $(g,n)=(0,4)$ and $(1,2)$ separately.

\begin{enumerate}
	\item[\textup{(1)}] Let $(g,n)=(0,4)$. The locus $Z$ consists of $(a_1,a_2,a_3)\in\M_{0,0}^3(\C)$ with $a_i$ nonzero and $[a_1,a_2a_3]=[a_2,a_3a_1]=[a_3,a_1a_2]=0$.	Suppose that one of $a_1a_2$, $a_2a_3$, and $a_3a_1$ is nonzero; say $a_1a_2\neq0$. Then we must have $\tr(a_1a_2)\neq0$ by Lemma \ref{a7}.(1). But the condition $[a_3,a_1a_2]=0$ implies that $a_1a_2$ is a scalar multiple of $a_3$ and hence $\tr(a_1a_2)=0$, a contradiction. Arguing similarly for the other cases, we conclude that
	$$a_1a_2=a_2a_3=a_3a_1=0.$$
	The condition $\tr a_2=\tr a_3=0$ then implies that $a_2$ and $a_3$ are both scalar multiples of $a_1$. Thus, $Z$ is $4$-dimensional.
	\item[\textup{(2)}] Let $(g,n)=(1,2)$. The locus $Z$ consists of $(a_1,a_2,a_3)\in\M_0^2\times\M_{0,0}(\C)$ with $a_i$ nonzero satisfying, among other things, $[\langle a_1,a_2\rangle,a_3]=0$.	By Lemma \ref{a2}.(1), this implies that $\langle a_1,a_2\rangle$ is a scalar multiple of $a_3$ and in particular $\tr\langle a_1,a_2\rangle=0$. The condition $\tr(\langle a_1,a_2\rangle a_3)=0$ then shows that in fact $\langle a_1,a_2\rangle a_3=0$ by Lemma \ref{a7}(1). The sublocus of $A$ where $\langle a_1,a_2\rangle\neq0$ has dimension
	$$2+1+3=6.$$
	Indeed, $a_3$ varies over a locus of dimension $2$, for fixed $a_3$ the value of $\langle a_1,a_2\rangle$ varies over a locus of dimension $1$ by Lemma \ref{a2}, and for fixed value of $\langle a_1,a_2\rangle\neq0$ the pair $(a_1,a_2)$ varies over a locus of dimension $3$ by Lemma \ref{6d3}.
	
	It remains to consider the locus $\langle a_1,a_2\rangle=0$. This condition implies that at least one of $a_1a_2, a_2a_1^*,a_1^*a_2^*$ is zero, by Lemma \ref{a6}.(2). The locus $(a_1,a_2)$ where at least two of them are zero has dimension at most $4$ by Lemma \ref{a8}, and hence $(a_1,a_2,a_3)$ would vary over a locus of dimension at most $4+2=6$, as desired. Thus, we may assume that exactly one of $a_1a_2, a_2a_1^*,a_1^*a_2^*$ is zero. We thus have the following possibilities.
	\begin{enumerate}
		\item Consider $a_1a_2=0$ and $a_2a_1^*,a_1^*a_2^*\neq0$. We must have $a_2a_1^*a_2^*\neq0$ by Lemma \ref{a6}.(1), and $a_1\neq0$. The condition
	$$[a_2a_1^*a_2^*,a_3a_1]=0$$
	shows that $\tr(a_3a_1)=0$ since $\tr(a_2a_1^*a_2^*)=0$. For fixed $a_3$, the locus of $a_1\in\M_0(\C)$ with $\tr(a_3a_1)=0$ has dimension $2$ (as easily verified when $a_3=\left[\begin{smallmatrix}0 & 1\\ 0 & 0\end{smallmatrix}\right]$, and deduced from this in the other cases by conjugation). For fixed nonzero $a_1$ the locus of of $a_2\in\M_0(\C)$ with $a_1a_2=0$ has dimension $2$ by Lemma \ref{a5}(1). Thus, the locus of $(a_1,a_2,a_3)\in Z(\C)$ with $a_1a_2=0$ has dimension at most $2+2+2=6$, as desired.
		\item Consider $a_2a_1^*=0$ and $a_1a_2,a_1^*a_2^*\neq0$, so that $(a_1,a_2)$ varies over a locus of dimension at most $5$. We have $a_1(\tr(a_2)\mathbf1-a_2)=(a_2a_1^*)^*=0$ which implies that $\tr a_2\neq0$ and similarly $\tr a_1\neq0$. Now, we have
		$$a_1^*a_2^*a_3a_1a_2=-[a_2,a_1^*a_2^*a_3a_1]=0.$$
	Since $a_1^*a_2^*,a_1a_2\neq0$ and $\tr(a_3)=0$, the above condition implies that, for fixed $(a_1,a_2)$, $a_3$ varies over a locus of dimension at most $1$ by Lemma \ref{a5}.(3). Hence, $(a_1,a_2,a_3)$ varies over a locus of dimension $5+1=6$, as desired.
	\item Consider $a_1^*a_2^*=0$. We have $a_2^*\neq0$ and $a_1a_2a_1^*\neq0$, and from the condition $[a_2^*a_3,a_1a_2a_1^*]=0$ we argue as in (a) to see that this locus has dimension at most $6$, as desired.
		\end{enumerate}
		Thus, we have shown that $Z$ has dimension $6$.
\end{enumerate}
In the remainder of the proof, we thus assume that $2g+n-1\geq4$. Let us use the notations of Section \ref{sect:4.1}. For a cyclic interval $I$ in $\Cfrak_{g+n-1}$, let $Z_I\subset Z$ be the sublocus defined by the condition $b_I\neq0$. Note that if $I\subseteq J$ are two cyclic intervals (containment meaning that $I$ is a subsequence of $J$) then $Z_J\subseteq Z_I$. Let
$$W_I=Z_I\setminus\bigcup_{I\subsetneq J}Z_J$$
where the union on the right hand side runs over all cyclic intervals $J$ in $\Cfrak_{g+n-1}$ containing $I$ and distinct from $I$. Note that we have a stratification
$$Z=\coprod_I W_I$$
where the union runs over the finitely many cyclic intervals $I$ of $\Cfrak_{g+n-1}$. Hence, it suffices to show that each $W_I$ has dimension at most $(6g+2n-3)-1$.

First, consider the case where $|I|=g+n-1$, and write $I=(i_0,\dots,i_0+g+n-2)$. The condition $[b_{i_0},b^{i_0}]=0$ on $Z$ with $b_{i_0},b^{i_0}\neq0$ then implies that, by Lemmas \ref{a2} and \ref{6d3}, the dimension of $W_I$ is at most
\begin{align*}
&6+1+6(g-1)+2(n-1)-3=(6g+2n-3)-1&&\text{if $g\geq1$, or}\\
&2+1+2(n-2)-3=(6g+2n-3)-1&&\text{if $g=0$}
\end{align*}
as desired.

Now, suppose that $|I|<g+n-1$. Let us further stratify $W_I=W_I'\sqcup W_I''$ where $W_I'$ (resp.~$W_I''$) is the sublocus defined by $b^I\neq0$ (resp.~$b^I=0$). The subloci $W_I'$ can be handled as above and shown to have dimension at most $(6g+2n-3)-1$. Thus, it remains to consider the subloci $W_I''$ for cyclic intervals $I$ with $|I|<g+n-1$. There are two possibilities to consider.
\begin{enumerate}
	\item[\textup{(1)}] We have $|I^c|=1$. Since $Z\subset H_{g,n}\setminus Y$, we must then have $I^c=\{i\}$ for some $i\in\{1,\dots,g\}$. Since $b_i=\langle a_{2i-1},a_{2i}\rangle=0$ on $W_I''$, by Lemma \ref{a6} at least one of $a_{2i-1}a_{2i}$, $a_{2i}a_{2i-1}^*$, and $a_{2i-1}^*a_{2i}^*$ must be zero. By Lemma \ref{a8}, the sublocus of $W_I''$ where at least two of the products is zero has dimension
	$$4+6(g-1)+2(n-1)=(6g+2n-3)-1.$$
	Thus, we are left to consider the sublocus $W_I'''$ of $W_I''$ where exactly one of the products $a_{2i-1}a_{2i}$, $a_{2i}a_{2i-1}^*$, and $a_{2i-1}^*a_{2i}^*$ is be zero. We thus have the following possibilities.
	\begin{enumerate}
		\item Consider $a_{2i-1}a_{2i}=0$ and $a_{2i}a_{2i-1}^*,a_{2i-1}^*a_{2i}^*\neq0$. We must then have $a_{2i}a_{2i-1}^*a_{2i}^*\neq0$ by Lemma \ref{a6}, and $a_{2i-1}\neq0$. The condition
	$$[a_{2i}a_{2i-1}^*a_{2i}^*,b^ia_{2i-1}]=0$$
	shows that $\tr(b^ia_{2i-1})=0$ since $\tr(a_{2i}a_{2i-1}^*a_{2i}^*)=0$. For fixed $b^i=b_I$ (which is nonzero on $W_I$), we see that the locus of $a_{2i-1}\in\M_0(\C)$ with $\tr(b^ia_{2i-1})=0$ has dimension $2$, and for fixed nonzero $a_{2i-1}$ the locus of $a_{2i}\in\M_0(\C)$ with $a_{2i-1}a_{2i}=0$ has dimension at most $2$. Thus, the locus of $(a_1,\dots,a_m)\in W_I''(\C)$ with $a_{2i-1}a_{2i}=0$ has dimension at most
	$$2+2+6(g-1)+2(n-1)=(6g+2n-3)-1$$
	as desired.
		\item Consider $a_{2i}a_{2i-1}^*=0$ and $a_{2i-1}a_{2i},a_{2i-1}^*a_{2i}^*\neq0$, so that $(a_{2i-1},a_{2i})$ varies over a locus of dimension at most $5$. Now, we have
		$$a_{2i-1}^*a_{2i}^*b^ia_{2i-1}a_{2i}=-[a_{2i-1}^*,a_{2i}^*b^ia_{2i-1}a_{2i}]=0.$$
		Since $a_{2i-1}^*a_{2i}^*,a_{2i-1}a_{2i}\neq0$, the above condition implies that, for fixed $(a_1,a_2)$, the value of $b^i$ varies over a locus of $2$ by Lemma \ref{a5}(3). Hence, by Lemma \ref{6d3}, the the sublocus of $W_I''$ with $a_{2i}a_{2i-1}^*=0$ has dimension at most
		$$5+2+6(g-1)+2(n-1)-3=(6g+2n-3)-1$$
		as desired.
		\item Consider $a_{2i-1}^*a_{2i}^*=0$ and $a_{2i-1}a_{2i},a_{2i}a_{2i-1}^*\neq0$. We must then have $a_{2i-1}a_{2i}a_{2i-1}^*\neq0$ by Lemma \ref{a6}, and $a_{2i}^*\neq0$. From the condition $[a_{2i}^*b^i,a_{2i-1}a_{2i}a_{2i-1}^*]=0$ we argue as in (a) to see that this locus has dimension at most $(6g+2n-3)-1$, as desired.
	\end{enumerate}
	\item[\textup{(2)}] We have $|I^c|\geq 2$. Let us write $I^c=\{i_0,\dots,i_0+k\}$ for some $k\geq1$. Let $J=\{i_0+1,\dots,i_0+k\}$, so that we have, on $W_I''$,
	$$b_{i_0}b_{J}=0,\quad b_Ib_{i_0}=0,\quad b_Jb_I=0.$$
	Let us further stratify $W_I''=V_I\sqcup V_I'$ where $V_I$ (resp.~$V_I'$) is the sublocus defined by the condition $b_J\neq 0$ (resp.~$b_J=0$). It remains to show that $V_I$ and $V_I'$ each have dimension at most $(6g+2n-3)-1$.
	\begin{enumerate}
		\item Let us define
		\begin{align*}
		g_1&=|I\cap\{1,\dots,g\}|,\\
		n_1&=|I\cap\{g+1,\dots,g+n-1\}|,\\
		g_2&=|J\cap\{1,\dots,g\}|,\quad \text{and}\\
		n_2&=|J\cap\{g+1,\dots,g+n-1\}|.
		\end{align*}
		Now, if $i_0\in\{1,\dots,g\}$, then by (a minor variant of) Lemmas \ref{6d3} and \ref{a5} and condition $b_{i_0}b_J=b_Ib_{i_0}=0$, $V_I$ has dimension bounded by
		$$6+(6g_1+2n_1-1)+(6g_2+2n_2-1)=6g+2(n-1)-2=(6g+2n-3)-1$$
		where $g_1+g_2=g-1$ and $n_1+n_2=n-1$. On the other hand, if $i_0\in\{g+1,\dots,g+n-1\}$, then by a similar argument $V_I$ has dimension bounded by
		$$2+(6g_1+2n_1-1)+(6g_2+2n_2-1)=6g+2(n-2)=(6g+2n-3)-1$$
		where $g_1+g_2=g$ and $n_1+n_2=n-2$. This proves that $V_I$ has dimension at most $(6g+2n-3)-1$, as desired.
		\item In addition to $g_i$ and $n_i$ as above, let us set
		\begin{align*}
		g'_1&=|(I\cup\{i_0\})\cap\{1,\dots,g\}|,\quad\text{and}\\
		n'_1&=|(I\cup\{i_0\})\cap\{g+1,\dots,g+n-1\}|.
		\end{align*}
		Since $g'_1+g_2=g$ and $n'_1+n_2=n-1$, by Lemma \ref{6d4} the dimension of $V'_I$ is bounded by
		$$(6g'_1+2n_1'-1)+(6g_2+2n_2-1)=6g+2(n-1)-2=(6g+2n-3)-1$$
		proving the desired result.
	\end{enumerate}
\end{enumerate}
This completes the proof that $\dim W_I''\leq (6g+2n-3)-1$. Therefore, we have $\dim Z\leq (6g+2n-3)-1$, completing the proof that the intersection $\Crit(E,F)\cap H_{g,n}$ has codimension at least $1$ in $H_{g,n}$. Finally, since $H_{g,n}$ is Cohen-Macaulay, the second statement of the proposition follows from the first.
\end{proof}

\section{Curves on surfaces} \label{sect:5}

The purpose of this section is to prove Theorem \ref{mtheorem2}. We first recall a result of Charles-March\'e \cite{cm} on multicurves in Section \ref{sect:5.1}, and give a refinement of it in Section \ref{sect:5.2} in the context of word length filtrations. We deduce Theorem \ref{mtheorem2} from this in Section \ref{sect:5.3}.

\subsection{Result of Charles-March\'e} \label{sect:5.1}
Let $\Sigma_{g,n}$ be a compact oriented surface of genus $g$ with $n\geq1$ boundary curves satisfying $2g+n-1\geq2$. Let us denote by $\chi_{g,n}$ the set of conjugacy classes of $\pi_1(\Sigma_{g,n})$, and let $\bar\chi_{g,n}$ be the quotient of $\chi_{g,n}$ obtained by setting the conjugacy class of $a\in\pi_1(\Sigma_{g,n})$ to be equivalent to that of $a^{-1}$. We shall refer to elements of $\bar\chi_{g,n}$ as reduced homotopy classes. We may identify $\chi_{g,n}$ with the set of free homotopy classes of loops $S^1\to \Sigma_{g,n}$, and $\bar\chi_{g,n}$ with the set of free homotopy classes of loops considered up to (possibly orientation-reversing) reparametrizations of $S^1$. Let $\Fin\bar\chi_{g,n}$ denote the collection of all finite multisets of elements in $\bar\chi_{g,n}$.

A \emph{curve} in $\Sigma_{g,n}$ will always be assumed simple and closed. A \emph{multicurve} in $\Sigma_{g,n}$ is a finite disjoint union of curves. A curve is \emph{nondegenerate} if it is not contractible, and \emph{essential} if it is nondegenerate and not isotopic to a boundary curve of $\Sigma_{g,n}$. A multicurve is nondegenerate, resp.~essential, if each of its components is.

Given a nondegenerate curve $a\subset \Sigma_{g,n}$, there is a well-defined element $a\in\bar\chi_{g,n}$ obtained by taking the reduced homotopy class of any choice of parametrization $S^1\to a$. This extends to a well-defined assignment
\begin{equation}
\label{assignment}
\{\text{isotopy classes of nondegenerate multicurves in $\Sigma_{g,n}$}\}\to\Fin\bar\chi_{g,n}
\end{equation}
sending $a=\coprod_i a_i\mapsto\{a_i\}$. Let $X_{g,n}$ be the character variety of $\Sigma_{g,n}$. In the coordinate ring $R_{g,n}=\C[X_{g,n}]$, we have $\tr_{a^{-1}}=\tr_{a}$ and $\tr_{aba^{-1}}=\tr_{b}$ for every $a,b\in\pi_1(\Sigma_{g,n})$. In particular, the function $\pi_1(\Sigma_{g,n})\to R_{g,n}$ given by $a\mapsto\tr_{a}$ factors through the projection $\pi_1(\Sigma_{g,n})\to\bar\chi_{g,n}$. We extend the function $\bar\chi_{g,n}\to R_{g,n}$ to an assignment
$$\Fin\bar\chi_{g,n}\to R_{g,n}$$
sending $a=\{a_1,\dots,a_s\}\mapsto \tr_{a}=\tr_{a_1}\dotsm \tr_{a_s}\in R_{g,n}$. In particular, for each isotopy class of a nondegenerate multicurve $a=\coprod_i a_i$ in $\Sigma_{g,n}$, there is a well-defined regular function $\tr_{a}=\prod_i\tr_{a_i}\in R_{g,n}$ (the empty multicurve $\emptyset$ corresponding to $1$). We have the following result due to Charles and March\'e \cite[Theorem 1.1]{cm}, which shows in particular that the assignment (\ref{assignment}) above is injective.

\begin{theorem}[Charles-March\'e \cite{cm}]
\label{cm}
The functions $\tr_{a}$, as $a$ runs over the isotopy classes of nondegenerate multicurves in $\Sigma_{g,n}$, form a $\C$-linear basis of $R_{g,n}$.
\end{theorem}

\subsection{Multicurves} \label{sect:5.2}
Let $\Sigma_{g,n}$ be a compact oriented surface of genus $g$ with $n\geq1$ boundary curves satisfying $m=2g+n-1\geq2$. We shall work with a presentation of $\Sigma_{g,n}$ as a ribbon graph, explicitly described as follows. Let $\Db=\{z\in\C:|z|\leq1\}$ be the closed unit disk with center $p_0$. Let us fix a sequence
$$\Ib_1,\Ib_2',\Ib_1',\Ib_2,\dots,\Ib_{2g-1},\Ib_{2g}',\Ib_{2g-1}',\Ib_{2g},\Ib_{2g+1},\Ib_{2g+1}',\dots,\Ib_{2g+n-1},\Ib_{2g+n-1}'$$
of disjoint closed intervals on the boundary $\del\Db$ of the disk, ordered clockwise. We shall also denote
$$\Ib=\bigcup_{i=1}^m(\Ib_k\cup\Ib_k').$$ We attach $2g+n-1$ rectangular strips $R_1,\dots,R_{2g+n-1}$ to $\del\Db$ so that $R_k$ joins $\Ib_k$ and $\Ib_k'$ in such a way that the resulting surface remains orientable. The resulting surface $\Sigma_{g,n}$ has genus $g$ and has $n\geq1$ boundary components; see Figure \ref{fig2} for an illustration in the case $(g,n)=(2,2)$.

\begin{figure}[ht]
    \centering
    \includegraphics{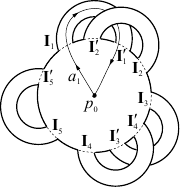}
    \caption{Ribbon graph presentation of $\Sigma_{2,2}$}
    \label{fig2}
\end{figure}

For $k=1,\dots,m$, let $a_k\in\pi_1(\Sigma_{g,n},p_0)$ be (the class of) a based loop which passes from $p_0$ to a point in $\Ib_k$ via a line segment in $\Db$, then to a point of $\Ib_k'$ via a simple path in $R_k$, then back to $p_0$ via a line segment in $\Db$. Finally, let $a_{2g+n}$ be the loop that is given by a line segment from $p_0$ to a point $q\in\del\Db$ lying between $\Ib_m'$ and $\Ib_1$, traveling around the boundary curve of $\Sigma_{g,n}$ containing $q$ once counterclockwise, and returning to $p_0$ via the same line segment joining $q$ to $p_0$. Under these choices, we have the standard presentation
$$\pi_1(\Sigma_{g,n})=\langle a_1,\dots,a_{2g+n}|[a_1,a_2]\dotsm[a_{2g-1},a_{2g}]a_{2g+1}\dotsm a_{2g+n}\rangle$$
and we set $\sigma=\{a_1,\dots,a_m\}$ to be the set of standard free generators. Now, given a collection of loops $a:\coprod_{i=1}^s S^1\to \Sigma_{g,n}$, there exists a homotopic immersion $a_0:\coprod_{i=1}^s S^1\to \Sigma_{g,n}$ such that
\begin{itemize}
	\item $\Img(a_0)$ intersects $\Ib$ only finitely many times,
	\item $\Img(a_0)$ contains no intersections outside of the interior of $\Db$, and
	\item $|\Img(a_0)\cap\Ib_k|=|\Img(a_0)\cap\Ib_k'|$ for every $k\in\{1,\dots,m\}$.
\end{itemize}
This motivates the following definition. By a \emph{chord} in $\Db$ we mean a straight line segment in $\Db$ joining two distinct points of $\del\Db$. Given two distinct points $v_1,v_2\in\del\Db$, we shall write $\{v_1,v_2\}$ to denote the chord joining them.

\begin{definition}
Let $r\geq0$ be an integer. We denote by $D_{g,n}(r)$ the collection of graphs $\Pi=(V_\Pi,E_\Pi)$ drawn (but not necessarily embedded) in $\Db$, consisting of a collection $E_\Pi=\{e_1,\dots,e_r\}$ of chords in $\Db$ for edges, such that the set of vertices $V_\Pi=\bigcup_{i=1}^r\del e_i$ consists of $2r$ distinct points lying in $\Ib\subset\del\Db$ and satisfying $|\Ib_k\cap V_\Pi|=|\Ib_k'\cap V_\Pi|$ for every $k=1,\dots,m$. We set $D_{g,n}=\bigcup_{r\geq0}D_{g,n}(r)$. Let $C_{g,n}(r)$ be the set of planar graphs in $D_{g,n}(r)$, and set $C_{g,n}=\bigcup_{r=0}^\infty C_{g,n}(r)$.
\end{definition}

A graph $\Pi\in D_{g,n}$ is considered only up to isotopy of the vertex set $V_\Pi$ within $\Ib$. By a generic choice of vertex positions, we may assume that each intersection point among the chords lies on exactly two chords, and denote by $I(\Pi)$ the number of intersection points on $\Pi$ under such configuration. Let us call $\Pi\in D_{g,n}$ \emph{reduced} if no edge of $\Pi$ joins two vertices lying on the same connected component of $\Ib$.

Given $\Pi\in D_{g,n}$, we obtain a finite	collection $a(\Pi)$ of loops in $\Sigma_{g,n}$ (considered up to reparametrization of each loop) by joining vertices of $\Pi$ in each $\Ib_k$ to $\Ib_k'$ by pairwise non-intersecting segments in $R_k$. This gives rise to an assignment
$$D_{g,n}\to\Fin\bar\chi_{g,n}$$
which is surjective by the previous discussion. If $\Pi\in C_{g,n}$, the construction $a(\Pi)$ gives us a multicurve in $\Sigma_{g,n}$, and the assignment $D_{g,n}\to\Fin\bar\chi_{g,n}$ restricts to
$$C_{g,n}\to\{\text{isotopy classes of multicurves in $\Sigma_{g,n}$}\}$$
which is surjective (but not injective in general).

The standard generating set $\sigma=\{a_1,\dots,a_m\}\subset\pi_1(\Sigma_{g,n})$ defines the word length function $\length_{\sigma}:\pi_1(\Sigma_{g,n})\to\Z_{\geq0}$, giving rise to $\length_\sigma:\overline\chi_{g,n}\to\Z_{\geq0}$ defined by
$$\length_\sigma(a)=\min\{\length_\sigma(b):\text{$b\in\pi_1(\Sigma_{g,n})$ lies in class $a$}\}.$$
We extend this additively to a function $\length_\sigma:\Fin\bar\chi_{g,n}\to\Z_{\geq0}$; namely, given $b=\{b_1,\dots,b_s\}\in\Fin\bar\chi_{g,n}$ we define $\length_\sigma(b)=\sum_{i=1}^s\length_\sigma(b_i)$. In particular, the length of a nondegenerate multicurve is well-defined via the assignment (\ref{assignment}). We remark that, by definition, the empty multicurve $\emptyset\subset \Sigma_{g,n}$ has length $\length_\sigma(\emptyset)=0$, and it is the unique nondegenerate multicurve with this property.

\begin{lemma}
\label{f}
For each isotopy class of a nondegenerate multicurve $a\subset \Sigma_{g,n}$, there is a unique reduced $\Pi\in C_{g,n}$ such that $a=a(\Pi)$. Furthermore, if $\length_\sigma(a)=r$, then $\Pi\in C_{g,n}(r)$.
\end{lemma}

\begin{proof}
Observe that, if $\Pi,\Pi'\in D_{g,n}$ are reduced and there is a homotopy from $a(\Pi)$ to $a(\Pi')$ within $\Sigma_{g,n}$ (up to reparametrizations of the loops), then there is a homotopy $(a(t))_{t\in[0,1]}$ with $a(0)=a(\Pi)$ and $a(1)=a(\Pi')$ such that
$$|\Img a(t)\cap\Ib|\leq|\Img a(\Pi)\cap\Ib|$$
for all $t\in[0,1]$. Similarly, given $\Pi,\Pi'\in C_{g,n}$ reduced such that there is an isotopy from $a(\Pi)$ to $a(\Pi')$ within $\Sigma_{g,n}$, there is an isotopy $(a(t))_{t\in[0,1]}$ from $a(\Pi)$ to $a(\Pi')$ satisfying the above inequality and in particular we must have $\Pi=\Pi'$. Further, given a nondegenerate multicurve $a$ it is clear that there is a reduced $\Pi\in C_{g,n}$ such that $a(\Pi)$ is isotopic to $a$. This shows that the assignment
$$\{\Pi\in C_{g,n}:\text{$\Pi$ reduced}\}\to\{\text{isotopy classes of nondegenerate multicurves in $\Sigma_{g,n}$}\}$$
given by $\Pi\mapsto a(\Pi)$ is bijective, proving the first statement.

To prove the second statement, note first that if $\Pi\in D_{g,n}(r)$ for some $r\geq0$ then $\length_\sigma(a(\Pi))\leq r$. Thus, given a nondegenerate multicurve $a\subset \Sigma_{g,n}$ with $\length_\sigma(a)=r$, by our observation in the previous paragraph it suffices to show that there exists $\Pi\in D_{g,n}(r)$ such that $a(\Pi)$ is homotopic to $a$. In fact, we are reduced to the case where $a$ is a curve. But for a curve, this statement is obvious.
\end{proof}

 We prove the following refinement of Theorem \ref{cm}. Recall that $\Fil^\sigma$ is the $\sigma$-word length filtration on $R_{g,n}$.

\begin{lemma}
\label{basislem}
The regular functions $\tr_{a}$, as $a$ runs over the isotopy classes of nondegenerate multicurves in $\Sigma_{g,n}$ with $\length_\sigma(a)\leq r$, form a $\C$-linear basis of $\Fil_r^\sigma R_{g,n}$.
\end{lemma}

\begin{proof}
By Theorem \ref{cm}, the collection $B_r\subset\Fil_r^\sigma$ of functions $\tr_{a}$, as $a$ runs over the isotopy classes of nondegenerate multicurves in $\Sigma_{g,n}$ with $\length_\sigma(a)\leq r$, is linearly independent over $\C$. Therefore, it remains to show that $B_r$ spans $\Fil_r^\sigma$. Given any $a\in\pi_1(\Sigma_{g,n})$ with $\length_\sigma(a)\leq r$, it is easy to see that there exists $\Pi\in D_{g,n}(s)$ with $s\leq r$ such that $a=a(\Pi)$. Thus, we are reduced to showing that $\tr_{a(\Pi)}\in\Span B_r$ for each $\Pi\in D_{g,n}(r)$.

So let $\Pi\in D_{g,n}(r)$ be given. We shall proceed by induction on the intersection number $I(\Pi)$. Let us choose the vertex positions for $\Pi$ so that each of the $I(\Pi)$ intersection points lies on exactly two of the chords. If $I(\Pi)=0$, then $\tr_{a(\Pi)}\in B_r$ and we are done. Suppose next that $I(\Pi)\geq1$, and let $e=\{v_1,v_2\}$ and $e'=\{v_1',v_2'\}$ be two edges of $\Pi$ intersecting at a point $p$. From this we can construct two graphs $\Pi',\Pi''\in D_{g,n}(r)$ as follows.
\begin{itemize}
	\item $V_{\Pi'}=V_{\Pi''}=V_{\Pi}$, and
	\item we have
	\begin{align*}
	&E_{\Pi'}=(E_{\Pi}\setminus\{e,e'\})\cup\{\{v_1,v_1'\},\{v_1',v_2'\}\},\\
	&E_{\Pi''}=(E_{\Pi}\setminus\{e,e'\})\cup\{\{v_1,v_2'\},\{v_1',v_2\}\}.
	\end{align*}
\end{itemize}
The fact that $\Pi,\Pi''\in D_{g,n}(r)$ implies that $\tr_{a(\Pi')},\tr_{a(\Pi'')}\in\Fil_r^\sigma$. We verify easily that $I(\Pi'),I(\Pi'')<I(\Pi)$, so by the inductive hypothesis $\tr_{a(\Pi')},\tr_{a(\Pi'')}\in\Span B_r$. But note that
$$\tr_{a(\Pi)}=\epsilon'\tr_{a(\Pi')}+\epsilon''\tr_{a(\Pi'')}$$
for some $\epsilon',\epsilon''\in\{\pm1\}$, as easily seen by the condition that $\tr_{a}\tr_{b}=\tr_{ab}+\tr_{ab^{-1}}$ in $R_{g,n}$ (applied after changing our base point to the intersection point $p$). We thus have $\tr_{a(\Pi)}\in\Span B_r$, completing the induction.
\end{proof}

\subsection{Proof of Theorem \ref{mtheorem2}} \label{sect:5.3}
Let $\Sigma_{g,n}$ be a surface of genus $g$ with $n\geq1$ boundary curves satisfying $2g+n-2>0$. We fix a standard presentation of $\pi_1(\Sigma_{g,n})$ and its standard set of free generators as in Section \ref{sect:5.2}. For $r\geq0$, let $c'_{g,n}(r)$ be the number of isotopy classes of nondgenerate multicurves $a\subset \Sigma_{g,n}$ with $\length_\sigma(a)=r$, and let $c_{g,n}(r)$ be the number of isotopy classes of essential multicurves $a\subset \Sigma_{g,n}$ with $\length_\sigma(a)=r$. Note that we have
$$(1-t)^{n-1}(1-t^{4g+n-1})\sum_{r=0}^\infty c_{g,n}'(r)t^r=\sum_{r=0}^\infty c_{g,n}(r)t^r.$$
Indeed, the $n$ boundary components of $\Sigma_{g,n}$ have lengths $1,\dots,1$, and $4g+n-1$, and any nondegenerate multicurve $a'$ in $\Sigma_{g,n}$ can be written uniquely as a disjoint union $a'=a\coprod a''$ in $\Sigma_{g,n}$ where $a$ is an essential multicurve and $a''$ is a finite disjoint union of curves each of which is isotopic to a boundary component of $\Sigma_{g,n}$. We are ready to prove Theorem \ref{mtheorem2}, restated below.

\begin{theorem}
The series $Z_{g,n}(t)=\sum_{r=0}^\infty c_{g,n}(r)t^r$ is a rational function, and
$$Z_{g,n}(1/t)=Z_{g,n}(t).$$
\end{theorem}

\begin{proof}
Note that $R_{g,n}^\sigma\simeq R_m^\sigma$ where $m=2g+n-1$ and $R_m$ is the coordinate ring of the character variety $X_m$ of the free group on $m$ generators $\sigma$. By Lemma \ref{basislem}, its Hilbert series $H_m(t)=\sum_{r=0}^\infty(\dim\Fil_r^\sigma)t^r$ is given by
$$H_m(t)=\frac{1}{1-t}\sum_{r=0}^\infty c_{g,n}'(r)t^r=\frac{Z_{g,n}(t)}{(1-t)^n(1-t^{4g+n-1})}.$$
By Theorem \ref{hilbert}, we have $H_m(1/t)=(-1)^{3m-2}t^{2m+1}H_m(t)$. We then have
\begin{align*}
Z_{g,n}(1/t)&=(1-t^{-1})^n(1-t^{-({4g+n-1})}) H_m(1/t)\\
&=(-1)^{3m-2}(1-t^{-1})^{n}(1-t^{-(4g+n-1)}) t^{2m+1} H_m(t)\\
&=(-1)^{3m-2+(n+1)}(1-t)^n(1-t^{4g+n-1})H_m(t)=(-1)^{3m-2+(n+1)}Z_{g,n}(t)
\end{align*}
noting that $2m+1=2(2g+n-1)+1=4g+2n-1$. But lastly, we note that $3m-2+(n+1)=3(2g+n-1)-2+(n+1)$ is even, and hence $Z_{g,n}(1/t)=Z_{g,n}(t)$, as desired.
\end{proof}

\begin{remark}
As the above proof shows, given $m\geq2$ the functional equation in Theorem \ref{hilbert} is equivalent to Theorem \ref{mtheorem2} for any surface $\Sigma_{g,n}$ with $m=2g+n-1$. In the next section, we shall give an independent combinatorial proof of Theorem \ref{mtheorem2} for the surface $\Sigma_{0,m+1}$, which then implies Theorem \ref{mtheorem2} in the general case by this observation.
\end{remark}

\section{Combinatorics of planar graphs} \label{sect:6}

Let $m\geq2$ be a fixed integer. Let $\Sigma_m=\Sigma_{0,m+1}$ be a compact oriented surface of genus $0$ with $m+1$ boundary components. Under the standard presentation of the fundamental group, the standard set $\sigma=\{a_1,\dots,a_m\}$ freely generates $\pi_1(\Sigma_m)\simeq F_m$. As in Section \ref{sect:5}, let us consider the generating series
$$Z_m(t)=Z_{0,m+1}(t)=\sum_{r=0}^\infty c_m(r)t^r$$
where $c_m(r)=c_{0,m+1}(r)$ denotes the number of essential multicurves $a\subset \Sigma_m$ with $\length_\sigma(a)=r$. In this section, we give a combinatorial proof of the following:

\begin{theorem}
The series $Z_m(t)$ is rational, and satisfies $Z_m(1/t)=Z_m(t)$.
\end{theorem}

We adopt the notations of Section \ref{sect:5}. In particular, $\Sigma_m$ is viewed as the closed unit disk $\Db$ with $2m$ rectangular strips suitably attached to the boundary $\del\Db$ along the $2m$ intervals $\Ib_1,\Ib_1',\dots,\Ib_m,\Ib_m'$. Let us write $D_m(r)=D_{0,m+1}(r)$ and $D_m=D_{0,m+1}$, and similarly $C_m(r)=C_{0,m+1}(r)$ and $C_m=C_{0,m+1}$.

Let $P_m=\{p_1,\dots,p_m\}$ be a collection of $m$ distinct points on $\del\Db$, with each $p_k$ lying on the component of $\del\Db\setminus\Ib$ between $\Ib_k$ and $\Ib_k'$. Given a reduced graph $\Pi\in C_m(r)$, note that one can identify the vertices in each $\Ib_k\cup\Ib_k'$ to a single vertex $p_k$, so as to obtain a planar multigraph in $\Db$ (except for self-loops which we draw outside of $\Db$) with vertices $P_m=\{p_1,\dots, p_m\}\subset\del\Db$ and $r$ edges such that every vertex has even degree. (See Figure \ref{fig3} for an example.) Conversely, given such a planar multigraph, we can construct a unique reduced graph in $C_m(r)$ that gives rise to it. Let $B_m(r)$ denote the set of planar multigraphs $\Pi$ without self-loops in $\Db$ having vertex set $P_m$ and $r$ edges such that every vertex in $\Pi$ has even degree. In light of Lemma \ref{f} and our discussion, we have
$$Z_m(t)=(1-t^m)\sum_{r=0}^\infty|B_m(r)|t^r.$$
Here, the factor $(1-t^m)$ is to cancel from our count the contribution of the $(m+1)$th boundary component of $\Sigma_m$, which has $\sigma$-length $m$.

\begin{figure}[ht]
    \centering
    \includegraphics{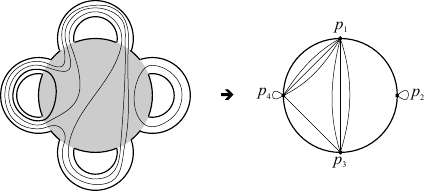}
    \caption{A multicurve in $\Sigma_4$ and associated planar multigraph}
    \label{fig3}
\end{figure}

Let us say that two points $p_i,p_j\in P_m$ are \emph{contiguous} if $i-j\in\{ 0,\pm1\}\modulo m$. For each integer $r\geq0$, let $A_m(r)$ denote the collection of planar multigraphs in $\Db$ with $r$ edges on the $m$ vertices $P_m$, such that no edge joins contiguous vertices. Let $A_m^s(r)$ be the set of multigraphs in $A_m(r)$ which are simple (i.e.~having no multiple edges). Note that we have $|A_m^s(r)|=0$ for $r>m-3$. Let $A_m=\bigcup_{r=0}^\infty A_m(r)$ and $A_m^s=\bigcup_{r=0}^\infty A_m^s(r)$. Define the generating series
$$F_m(t)=\sum_{r=0}^\infty |A_m(r)|t^r.$$

\begin{lemma}
\label{aslem}
We have $\sum_{r=1}^{m-3}(-1)^{r-1}|A_m^s(r)|=1+(-1)^{m-4}$.
\end{lemma}

\begin{proof}
The statement is obvious for $m\leq4$, and hence we treat the case $m\geq5$. We construct a simplicial complex $\Lb_m$ as follows. For each $r\geq0$, the $r$-simplices of $\Lb_m$ are labeled by $A_m^s(r+1)$; we let $\Delta(\Pi)$ denote the $r$-simplex associated to $\Pi\in A_m^s(r+1)$. Whenever a graph $\Pi\in A_m(r+1)$ is obtained by deleting a number of edges from some $\Pi'\in A_m(r'+1)$ with $r<r'$, we glue the corresponding $r$-simplex $\Delta(\Pi)$ to the $r'$-simplex $\Delta(\Pi')$ compatibly. From this construction, the lemma is just the statement
$$\chi(\Lb_m)=1+(-1)^{m-4}$$
where the left hand side is the Euler characteristic of $\Lb_m$. In fact, it is a relatively well known result (first published by Lee \cite[Theorem 1]{lee}, also found independently by Haiman (cf.~Lee \emph{loc.cit.})) that $\Lb_m$ is isomorphic to the boundary complex of a convex polytope of dimension $m-3$, and therefore the following stronger result holds:
$$H_0(\Lb_m,\Z)=H_{m-4}(\Lb_m,\Z)=\Z,\quad\text{and}\quad H_r(\Lb_m,\Z)=0\quad\forall 1\leq r<m-4.$$
This finishes the proof of the lemma.
\end{proof}

Let $r_0\leq m-3$ be a nonnegative integer, and let $\Pi\in A_m^s(r_0)$. For $r_0\leq r\leq m-3$, let us define $A_m^s(r,\Pi)=\{\Pi'\in A_m^s(r):\Pi\subseteq\Pi'\}$. Here, the containment $\Pi\subseteq\Pi'$ means that the edge set of $\Pi'$ contains the edge set of $\Pi$. Both graphs have the same vertex set, namely $P_m$.

\begin{corollary}
\label{altsumcor}
Let $r_0\leq m-3$ be a nonnegative integer, and $\Pi\in A_m^s(r_0)$. Then
$$\sum_{r=r_0}^{m-3}(-1)^{r-r_0}|A_m^s(r,\Pi)|=(-1)^{m-3-r_0}.$$
\end{corollary}

\begin{proof}
Note that $\Pi$ divides $\Db$ into a number of polygons with numbers of sides $m_1,\dots,m_{r_0+1}\geq3$ such that $\sum_{k=1}^{r_0+1}(m_k-3)+r_0=m-3$. Here, we are counting each arc of $\del\Db$ joining two contiguous vertices also as sides of a polygon. Using this and Lemma \ref{aslem}, we find
\begin{align*}
\sum_{r=r_0}^{m-3}(-1)^{r-r_0}|A_m^s(r,\Pi)|&=\prod_{k=1}^{r_0+1}\left\{1-\sum_{r=1}^{m_k-3}(-1)^{r-1}|A_{m_k}(r)|\right\}\\
&=\prod_{k=1}^{r_0+1}(-1)^{m_k-3}=(-1)^{m-3-r_0},
\end{align*}
which is the desired result.
\end{proof}

\begin{proposition}
\label{fprop}
The generating series $F_m(t)=\sum_{r=0}^\infty|A_m(r)|t^r$ satisfies
$$F_m(t)=\frac{f_m(t)}{(1-t)^{m-3}}$$
where $f_m(t)$ is a polynomial of degree $m-3$ such that $t^{m-3}f_m(1/t)=f_m(t)$.
\end{proposition}

\begin{proof}
For each multigraph $\Pi\in A_m(r)$, there is a unique simple graph $\Pi^s\in A_m^s$ that has the same adjacent vertices as $\Pi$ (i.e.~vertices joined by an edge in $\Pi$ are still joined in $\Pi^s$). Note that $\Pi$ is obtained from $\Pi^s$ by adding edges to adjacent vertices in $\Pi^s$. Thus, we have
$$F_m(t)=\sum_{\Pi\in A_m^s}\frac{t^{e(\Pi)}}{(1-t)^{e(\Pi)}}$$
where $e(\Pi)$ denotes the number of edges of $\Pi$. Since $|A_m^s(r)|=0$ for $r>m-3$ and $e(\Pi)\leq m-3$ for all $\Pi\in A_m^s$, we have $F_m(t)=f_m(t)/(1-t)^{m-3}$ where $f_m(t)=\sum_{r=0}^{m-3} |A_m^s(r)|t^{r}(1-t)^{m-3-r}$ is a polynomial of degree at most $m-3$. As we have $f_m(0)=F_m(0)=1$ from the interpretation of $F_m(t)$ as a generating function, to prove that $f_m(t)$ has degree exactly $m-3$ it suffices to prove the remaining assertion $f_m(t)=t^{m-3}f_m(1/t)$, or equivalently
$$f_m(t)=\sum_{r=0}^{m-3}|A_m^s(r)|(t-1)^{m-3-r}.$$
To prove this, we compute $F_m(t)$ in another way by a type of inclusion-exclusion principle. For each $\Pi\in A_m^s$, note that the generating function for the number of graphs $\Pi'\in A_N$ with $(\Pi')^s\subseteq\Pi$ is given by $1/(1-t)^{e(\Pi)}$. Thus, we may write $F_m(t)$ by first adding the contributions from $A_m^s(m-3)$, then adding contributions from $A_m^s(m-4)$ and subtracting away the corresponding ``overcount'' from the previous step, and so on:
$$F_m(t)=\frac{|A_m^s(m-3)|}{(1-t)^{m-3}}+\frac{1}{(1-t)^{m-4}}\sum_{\Pi\in A_m^s(m-4)}(1-|A_m^s(m-3,\Pi)|)+\cdots.$$
Therefore, we see upon reflection that
\begin{align*}
F_m(t)&=\sum_{r=0}^{m-3}\frac{1}{(1-t)^r}\sum_{\Pi\in A_m^s(r)}\left\{1+\sum_{r'=r+1}^{m-3}\sum_{k=1}^{r'-r}(-1)^k\binom{r'-r}{k-1}|A_m^s(r',\Pi)|\right\}\\
&=\sum_{r=0}^{m-3}\frac{1}{(1-t)^{r}}\sum_{\Pi\in A_m^s(r)}\sum_{r'=r}^{m-3}(-1)^{r'-r}|A_m^s(r',\Pi)|\\
&=\sum_{r=0}^{m-3}\frac{(-1)^{m-3-r}|A_m^s(r)|}{(1-t)^{r}}=\frac{1}{(1-t)^{m-3}}\sum_{r=0}^{m-3}|A_m^s(r)|(t-1)^{m-3-r}
\end{align*}
(where the first equality on the last line follows by Corollary \ref{altsumcor}) from which we obtain $f_m(t)=\sum_{r=0}^{m-3}|A_m^s(r)|(t-1)^{m-3-r}$, showing that $f_m$ has degree $m-3$ and that $t^{m-3}f(1/t)=f(t)$ as desired.
\end{proof}

Let $m\geq1$ be an integer. As introduced in the beginning of Section \ref{sect:6}, for each $r\geq0$ let $B_m(r)$ denote the set of planar multigraphs $\Pi$ without self-loops in the closed unit disk $\Db$ having vertices $P_m=\{p_1,\dots,p_m\}$ and $r$ edges such that every vertex in $\Pi$ has even degree. Let $B_m^s(r)$ denote the set of those $\Pi\in B_m(r)$ that are simple graphs, i.e.~between any two vertices there is at most one edge. Let $B_m=\bigcup_{r=0}^\infty B_m(r)$ and $B_m^s=\bigcup_{r=0}^\infty B_m^s(r)$. We are interested in the generating function
$$G_m(t)=\sum_{r=0}^\infty|B_m(r)|t^r.$$
By a \emph{bigon} in a multigraph we shall mean a pair of edges whose endpoints coincide. Note that for every $\Pi\in B_m$ there exists a unique $\Pi^s\in B_m^s$ such that $\Pi$ is obtained by adding bigons to $\Pi^s$ without losing planarity. More precisely, $\Pi^s$ is constructed from $\Pi$ by repeatedly removing bigons from $\Pi$ until none remain. Thus, $G_m(t)$ is the sum of the contribution of each $\Pi\in B_m^s$ which is analyzed as follows. Let $\Pi\in B_m^s(r)$ be given. Let $e(\Pi)$ denote the number of edges of $\Pi$, and $e^\circ(\Pi)$ the number of edges of $\Pi$ joining non-contiguous vertices. Note that the edges of $\Pi$ joining non-contiguous vertices gives a decomposition of $\Db$ into $e^\circ(\Pi)+1$ polygons (as in the proof of Corollary \ref{altsumcor}) with numbers of sides $m_1,\dots,m_{e^\circ(\Pi)+1}$ such that
$$\sum_{k=1}^{{e^\circ(\Pi)+1}}(m_k-3)+e^\circ(\Pi)=m-3.$$
In this setting, the contribution from $\Pi$ to the generating function $G_m(t)$ is
\begin{align*}
\frac{t^{e(\Pi)}}{(1-t^2)^{m+e^\circ(\Pi)}}\prod_{k=1}^{e^\circ(\Pi)+1} F_{m_k}(t^2)&=\frac{t^{e(\Pi)}}{(1-t^2)^{m+e^\circ(\Pi)}}\prod_{k=1}^{e^\circ(\Pi)+1}\frac{f_{m_k}(t^2)}{(1-t^2)^{m_k-3}}\\
&=\frac{t^{e(\Pi)}\prod_{k=1}^{e^\circ(\Pi)+1}f_{m_k}(t^2)}{(1-t^2)^{2m-3}}
\end{align*}
where each $f_{m_k}(t^2)$ is a polynomial of degree $2(m_k-3)$ by Proposition \ref{fprop}. Thus, defining $d(\Pi)=2\sum_{k=1}^{e^\circ(\Pi)+1}(m_k-3)$ and $g_\Pi(t)=\prod_{k=1}^{e^\circ(\Pi)+1}f_{m_k}(t^2)$ for each $\Pi\in B_m^s$ with the $m_k$'s as above, we see that $g_\Pi(t)$ is a polynomial of degree $d(\Pi)$ with the symmetry
$t^{d(\Pi)}g_\Pi(1/t)=g_\Pi(t)$. We may thus write
$$G_m(t)=\frac{1}{(1-t^2)^{2m-3}}\sum_{\Pi\in B_m^s}t^{e(\Pi)}g_\Pi(t),$$
and we are ready to prove our main result of this section.

\begin{theorem}
\label{gth}
The series $G_m(t)=\sum_{r=0}^\infty|B_m(r)|t^r$ satisfies the symmetry
$$G_m(1/t)=(-1)^{2m-3}t^{m}G_m(t).$$
\end{theorem}

\begin{proof}
Given $\Pi\in B_m^s$, we define its \emph{dual} $\Pi^\vee$ as the graph obtained from $\Pi$ by removing (resp.~adding) one edge between contiguous vertices that were adjacent (resp.~not adjacent) in $\Pi$, and retaining any edges between non-contiguous vertices in $\Pi$. It is easy to see that in fact $\Pi^\vee\in B_m^s$, and $(\Pi^\vee)^\vee=\Pi$. Note that we have $g_\Pi(t)=g_{\Pi^\vee}(t)$ and $d(\Pi)=d(\Pi^\vee)$; these follow from the fact that the construction of $g_\Pi(t)$ only depended on the edges of $\Pi$ joining non-contiguous vertices. Furthermore, note that we have
$$d(\Pi)+e(\Pi)+e(\Pi^\vee)=3m-6.$$
Hence, writing
$$G_m(t)=\frac{1}{2}\frac{1}{(1-t^2)^{2m-3}}\sum_{\Pi\in B_m^s}(t^{e(\Pi)}+t^{e(\Pi^\vee)})g_{\Pi}(t),$$
we obtain
\begin{align*}
G_m(1/t)&=\frac{1}{2}\frac{1}{(1-t^{-2})^{2m-3}}\sum_{\Pi\in B_m^s}(t^{-e(\Pi)}+t^{-e(\Pi^\vee)})g_{\Pi}(1/t)\\
&=\frac{1}{2}\frac{(-1)^{2m-3}t^{4m-6}}{(1-t^2)^{2m-3}}\sum_{\Pi\in B_m^s}\frac{(t^{e(\Pi^\vee)}+t^{e(\Pi)})g_{\Pi}(t)}{t^{3m-6}}=(-1)^{2m-3}t^{m}G_m(t)
\end{align*}
which gives us the result.
\end{proof}

Returning to the beginning of Section \ref{sect:6}, we have $Z_m(t)=(1-t^m)G_m(t)$. Hence, using Theorem \ref{gth}, we find that
\begin{align*}
Z_m(1/t)&=(1-t^{-m})G_m(1/t)=(1-t^{-m})(-1)^{2m-3}t^mG_m(t)\\
&=(1-t^m)G_m(t)=Z_m(t),
\end{align*}
which concludes the proof of Theorem 6.1.

\appendix
\section{Auxiliary results on matrices} \label{sect:a}

\subsection{Identities} \label{sect:a.1}
Let $\M$ be the complex affine scheme parametrizing $2\times2$ matrices. Let $x$ denote the standard matrix variable for $\M$. Viewing $x$ and $x^*$ as $2\times 2$ matrices with coefficients in $\C[\M]$, we have
\begin{align}
x+x^*=\tr(x)\mathbf 1,\quad\text{and}\quad xx^*=x^*x=\det(x)\mathbf1.
\end{align}
Note that $\det(a^*)=\det(a)$ and $\tr(a^*)=\tr(a)$ for any $2\times2$ matrix $a$. Multiplying the matrix identity $b+b^*=\tr(b)\mathbf 1$ by the matrix $a$ and taking the trace, we also obtain the identity
\begin{align}
\tr(a)\tr(b)=\tr(ab)+\tr(ab^*).
\end{align}

\subsection{Invariant theory} \label{sect:a.2}
We have an action of $\SL_2$ on $\M$ by conjugation, and hence a diagonal conjugation action of $\SL_2$ on $\M^2$. It is classical that the ring of invariants $\C[\M^2\git\SL_2]=\C[\M^2]^{\SL_2}$ is a polynomial ring on $5$ generators $$\C[\tr(x_1),\tr(x_2),\tr(x_1x_2),\det(x_1),\det(x_2)],$$
where $x_1$ and $x_2$ are the standard matrix variables on $\M^2$. (See for example \cite[Theorem 5.3.1(ii), p.68]{df}) In particular, any $\SL_2$-invariant regular function in matrix variables $x_1$ and $x_2$ is a polynomial combination of the functions $\tr(x_1)$, $\tr(x_2)$, $\tr(x_1x_2)$, $\det(x_1)$, and $\det(x_2)$. For example, using the identities (1) and (2), we find that
\begin{align*}
\tr(\langle a_1,a_2\rangle)=&\tr(a_1)^2\det(a_2)+\tr(a_2)^2\det(a_1)+\tr(a_1a_2)^2\\
&\quad -\tr(a_1)\tr(a_2)\tr(a_1a_2)-2\det(a_1)\det(a_2).
\end{align*}
For each $t\in\C$, the action of $\SL_2$ on $\M^2$ preserves the closed subscheme $\M_t^2\subset\M^2$, and since $\SL_2$ is linearly reductive we have
$$\C[\M_t^2\git\SL_2]=\C[\M^2]^{\SL_2}/(\det(x_1)-t,\det(x_2)-t)\simeq\C[\tr(x_1),\tr(x_2),\tr(x_1x_2)].$$
In particular, $\SL_2^2\git\SL_2\simeq\A^3$, which is due to Fricke (see Goldman \cite[Section 2.2]{goldman2} for details).
\begin{lemma}
\label{a1}
The quotient morphism $\M^2\to\M^2\git\SL_2\simeq\A^5$ is surjective.
\end{lemma}
\begin{proof}
By general theory, $\C[\M^2]^{\SL_2}$ is a pure subring of $\C[\M^2]$. In particular, for any ideal $I$ of $\C[\M^2]^{\SL_2}$ we have $I\C[\M^2]\cap \C[\M^2]^{\SL_2}=I$. In particular, if $I$ is a proper ideal of $\C[\M^2]^{\SL_2}$ then $I\C[\M^2]$ is a proper ideal of $\C[\M^2]$. This implies the lemma.
\end{proof}

\subsection{Commutators} \label{sect:a.3}
\begin{lemma}
\label{a2}
We have the following.
\begin{enumerate}
	\item[\textup{(1)}] For $b\in\M_0(\C)$ nonscalar, $\{a\in\M_0(\C):[a,b]=0\}$ has dimension $1$.
	\item[\textup{(2)}] For $b\in\SL_2(\C)$ nonscalar and $k\in\C$, the locus of $a\in\SL_2(\C)$ with $[a,b]=0$ and $\tr(ab)=k$ has dimension $1$ if $\tr(b),k\in\{\pm2\}$, and is finite otherwise.
\end{enumerate}
\end{lemma}

\begin{proof}
Since $b$ is nonscalar, any $a\in\M(\C)$ with $[a,b]=0$ must be of the form $a=\lambda_1\mathbf1+\lambda_2 b$ with $\lambda_i\in\C$. In such a case, we have
\begin{align*}
\tag{A}\det(a)&=\det(\lambda_1\mathbf1+\lambda_2b)=\lambda_1^2+\lambda_1\lambda_2\tr(b)+\lambda_2^2\det(b),\\
\tag{B}\tr(ab)&=\tr((\lambda_1\mathbf1+\lambda_2b)b)=\lambda_1\tr(b)+\lambda_2\tr(b^2).
\end{align*}
We first prove (1). If $\det(a)=\det(b)=0$, then we have $\lambda_1^2+\lambda_1\lambda_2\tr(b)=0$. It follows that the desired locus is the union $\{\lambda b:\lambda\in\C\}\cup\{\lambda b^*:\lambda\in\C\}$, and hence is $1$-dimensional. It remains to prove (2). Let $S$ be the locus of $a\in\SL_2(\C)$ determined by the conditions of (2). Under the assumptions, by the Cayley-Hamilton theorem we have $\tr(b^2)-\tr(b)^2+2=0$ and hence $\tr(b^2)\neq0$ or $\tr(b)^2\neq0$. Thus, equation (B) defines a line in the $(\lambda_1,\lambda_2)$-plane. In particular, $S$ is at most $1$-dimensional. Suppose it is not finite. The conic in the $(\lambda_1,\lambda_2)$-plane defined by equation (A) must then be degenerate, i.e.~the discriminant
$$-\det(a)\det\begin{bmatrix}1 & \tr(b)/2 \\ \tr(b)/2 &\det(b)\end{bmatrix}=-1(1-\tr(b)^2/4)$$
is zero. Thus, we must have $4=\tr(b)^2$, i.e.~$\tr(b)=\pm2$. Equations (A) and (B) then become
\begin{align*}
\tag{A$'$}1&=(\lambda_1+\lambda_2\tr(b)/2)^2,\\
\tag{B$'$}k&=\tr(b)(\lambda_1+\lambda_2\tr(b)/2).
\end{align*}
The degenerate conic defined by equation (A$'$) is a union of two disjoint lines. For $S$ to be infinite, one of the two lines must coincide with the line defined by equation (B$'$). In other words, we must have $(k/\tr(b))^2=1$, or $k=\pm\tr(b)\in\{\pm2\}$. If this happens, then $S$ is one-dimensional. Thus, we have proved the lemma.
\end{proof}

\subsection{Matrices of determinant one} \label{sect:a.4}
\begin{lemma}
\label{a3}
Fix $(k_1,k_2,k_3)\in\C^3$. The morphism:
\begin{enumerate}
	\item[\textup{(1)}] $\prod_{i=1}^2\SL_{2,k_i}\to\SL_2$ given by $(a_1,a_2)\mapsto a_1a_2$ is surjective over $\SL_2\setminus\{\pm\mathbf1\}$.
	\item[\textup{(2)}] $\prod_{i=1}^3\SL_{2,k_i}\to\SL_2$ given by $(a_1,a_2,a_3)\mapsto a_1a_2a_3$ is surjective.
	\item[\textup{(3)}] $\langle-,-\rangle:\SL_2^2\to\SL_2$ is surjective.
\end{enumerate}
\end{lemma}
\begin{proof}
(1) By Lemma \ref{a1}, given any $b\in\SL_2(\C)$ there exist $a_i\in\SL_{2,k_i}(\C)$ such that $\tr(a_1a_2)=\tr(b)$. If $\tr(b)\neq\pm2$, then this implies $(ga_1g^{-1})(ga_2g^{-1})=b$ for some $g\in\SL_2(\C)$. Consider now the case $\tr(b)=2s$ for some $s\in\{\pm1\}$ but $b\neq s\mathbf 1$. If $k_1\neq sk_2$, then $a_1a_2$ cannot be $s\mathbf1$ and hence $(ga_1g^{-1})(ga_2g^{-1})=b$ for some $g\in\SL_2(\C)$ as before. The case remains that $k_1=sk_2=k$. Let $\lambda\in\C^{\times}$ be a root of the polynomial $x^2-kx+1=0.$ We then have
$$\begin{bmatrix}\lambda & 1\\0 &\lambda^{-1}\end{bmatrix}\begin{bmatrix}s\lambda^{-1} & s\\0 &s\lambda\end{bmatrix}=\begin{bmatrix}s & 2s\lambda\\0 & s\end{bmatrix},$$
and letting $a_1$ and $a_2$ respectively be the two matrices on the left hand side, there exists $g\in\SL_2(\C)$ such that $(ga_1g^{-1})(ga_2g^{-1})=b$.

(2) By part (1), given any $b\in\SL_2(\C)$ different from $\pm\mathbf1$ and any $k'\in\C\setminus\{\pm2\}$, there exist $b'\in\SL_{2,k'}(\C)$ and $a_3\in\SL_{2,k_3}(\C)$ such that $b'a_3=b$. Since $b'\neq\pm\mathbf1$ due to the condition $k'\neq\pm2$, again by part (1) there exist $a_1\in\SL_{2,k_1}(\C)$ and $a_2\in\SL_{2,k_2}(\C)$ such that $a_1a_2=b'$, and thus $a_1a_2a_3=b$, as desired. If $b=s\mathbf1$ for some $s\in\{\pm1\}$, then choosing $a_3\in\SL_{2,k_3}(\C)$ different from $\pm\mathbf1$, there exist $a_1\in\SL_{2,k_1}(\C)$ and $a_2\in\SL_{2,k_2}(\C)$ such that $a_1a_2=sa_3^{-1}$ and hence $a_1a_2a_3=b$.

(3) From Lemma \ref{a1}, given any $b\in\SL_2(\C)$ there exist $a_1,a_2\in\SL_2(\C)$ such that $\tr\langle a_1,a_2\rangle=\tr b$. If $b\neq\pm\mathbf1$, then $\langle ga_1g^{-1},ga_2g^{-1}\rangle=g\langle a_1,a_2\rangle g^{-1}=b$ for some $g\in\SL_2(\C)$, as desired. Finally, in the case where $b=\pm\mathbf1$, we have
$$\langle\mathbf1,\mathbf1\rangle=\mathbf1\quad\text{and}\quad\left\langle\begin{bmatrix}i & 0\\ 0 & -i\end{bmatrix},\begin{bmatrix}0 & 1\\ -1 & 0\end{bmatrix}\right\rangle=-\mathbf1$$
which proves the desired result.
\end{proof}

\begin{lemma}
\label{a4}
The fiber of the morphism $\langle-,-\rangle:\SL_2^2\to\SL_2$ above $\mathbf1$ has dimension $4$. The fiber above $-\mathbf1$ has dimension $3$.
\end{lemma}
\begin{proof}
For $s\in\{+,-\}$, let $F_s$ be the fiber of $\langle-,-\rangle$ over $s\mathbf1$. Consider the projection
$$\pi:F_{\pm}\subset\SL_2^2\to\SL_2$$
onto the second factor. We first consider $F_+$. The projection $\pi:F_+\to\SL_2$ is surjective, since for any $b\in\SL_2(\C)$ we have $\langle\mathbf1,b\rangle=\mathbf1$. Given any $b\in\SL_2(\C)$, we identify the fiber $\pi^{-1}(b)$ with the scheme of matrices $a\in\SL_2(\C)$ such that $\langle a,b\rangle=\mathbf1$, or in other words $[a,b]=0$. If $b=\pm\mathbf1$, then $\pi^{-1}(b)=\SL_2$ has dimension $3$. If $b\neq\pm\mathbf1$, then any $a\in\pi^{-1}(b)(\C)$ must be of the form $a=\lambda_1\mathbf1+\lambda_2 b$ for some $\lambda_1,\lambda_2\in\C$. The condition $\det(a)=1$ then identifies $\pi^{-1}(b)$ with a curve in the $(\lambda_1,\lambda_2)$-plane. Hence, the fibers of $\pi:F_+\to\SL_2$ above $\SL_2\setminus\{\pm\mathbf1\}$ are $1$-dimensional. Therefore $F_+$ is $3+1=4$-dimensional.

Consider next $F_-$. Given $b\in\SL_2(\C)$, we identify the fiber $\pi^{-1}(b)$ with the scheme of matrices $a\in\SL_2(\C)$ such that $\langle a,b\rangle=-\mathbf1$, or in other words $ab+ba=0$. Writing $a=(a_{ij})$ and $b=(b_{ij})$, the condition $ab+ba=0$ amounts to
$$\begin{bmatrix}2b_{11} & b_{21} & b_{12} & \\ b_{12} & b_{11}+b_{22} & & b_{12}\\ b_{21} & & b_{11}+b_{22} & b_{21}\\ & b_{21} & b_{12} & 2b_{22}\end{bmatrix}\begin{bmatrix}a_{11}\\ a_{12} \\ a_{21} \\a_{22}\end{bmatrix}=0.$$
Hence, for the fiber $\pi^{-1}(b)$ to be nonempty, we need the determinant of the $4\times 4$ matrix on the left hand side to be zero, or in other words $\tr(b)=b_{11}+b_{22}=0$. Thus, the image of $\pi$ is $\SL_{2,0}$. Given any $b\in\SL_{2,0}(\C)$, the $4\times 4$ matrix above has rank $2$, and hence its kernel is $2$-dimensional. Given the additional determinant condition $a_{11}a_{22}-a_{12}a_{21}=1$, the fiber $\pi^{-1}(b)$ has dimension $1$. Since $\SL_{2,0}$ is $2$-dimensional, we see that $F_-$ is $2+1=3$-dimensional.
\end{proof}

\subsection{Matrices of determinant zero} \label{sect:a.5}
\begin{lemma}
\label{a5}
Let $b,b'\in\M_0(\C)$ both be nonzero.
\begin{enumerate}
	\item[\textup{(1)}] The locus of $a\in\M_0(\C)$ with $ab=0$ or $ba=0$ is $2$-dimensional.
	\item[\textup{(2)}] The locus of $a\in\M_0(\C)$ with $ab=ba=0$ is $1$-dimensional.
	\item[\textup{(3)}] The locus of $a\in\M_0(\C)$ with $bab'=0$ is $2$-dimensional. The intersection of this locus with $\M_{0,0}$ is at most $1$-dimensional.
\end{enumerate}
\end{lemma}

\begin{proof}
We first prove (1) and (2). Let $a=(a_{ij})\in\M(\C)$. Without loss of generality, after conjugation we may assume that
$$b=\begin{bmatrix}\lambda & 0\\ 0 & 0\end{bmatrix}\text{ for some $\lambda\in\C^\times$},\quad\text{or}\quad b=\begin{bmatrix}0 & 1\\ 0 & 0\end{bmatrix}.$$
In the former case: $ab=0$ if and only if $a_{11}=a_{21}=0$, and $ba=0$ if and only if $a_{11}=a_{12}=0$. In the latter case: $ab=0$ if and only if $a_{11}=a_{21}=0$, and $ba=0$ if and only if $a_{21}=a_{22}=0$. In both cases, if $ab=0$ or $ba=0$ then we automatically have $\det(a)=0$. Parts (1) and (2) follow immediately from these.

(3) Let us write $b=(b_{ij})$, and $b'=(b_{ij}')$. The linear map $a\mapsto bab'$ on $\M(\C)$ is given in terms of matrix coefficients $a=(a_{ij})$ by
$$\begin{bmatrix}a_{11}\\a_{12}\\a_{21}\\a_{22}\end{bmatrix}\mapsto\begin{bmatrix}
b_{11}b_{11}'&b_{11}b_{21}'&b_{12}b_{11}'&b_{12}b_{21}'\\
b_{11}b_{12}'&b_{11}b_{22}'&b_{12}b_{12}'&b_{12}b_{22}'\\
b_{21}b_{11}'&b_{21}b_{21}'&b_{22}b_{11}'&b_{22}b_{21}'\\
b_{21}b_{12}'&b_{21}b_{22}'&b_{22}b_{12}'&b_{22}b_{22}'
\end{bmatrix}\begin{bmatrix}a_{11}\\a_{12}\\a_{21}\\a_{22}\end{bmatrix}$$
and since $b,b'\in\M_0(\C)$ are both nonzero, the $4\times 4$ matrix above has rank $1$. Thus, the kernel of the linear map $a\mapsto bab'$ on $\M(\C)$ is a linear subspace of dimension $3$. Since the hypersurface $\M_0(\C)\subset\M(\C)$ is integral (and not linear), its intersection with the above kernel has dimension at most $2$, proving the first assertion.

We now prove the last assertion. We claim that the linear subspace $V$ consisting of $a\in\M(\C)$ satisfying $\tr(a)=0$ and $bab'=0$ is $2$-dimensional. Indeed, otherwise the second and third columns of the $4\times 4$ matrix above must be identically zero, contradicting the hypothesis that $b$ and $b'$ are both nonzero matrices. Now, since $V$ and $\M_{0,0}$ are both integral subschemes of dimension $2$ in $\M$, they must coincide or have intersection with dimension at most $1$. But $\M_{0,0}$ is not linear, and hence $V\cap\M_{0,0}$ has dimension at most $1$, proving the last assertion.
\end{proof}

\begin{lemma}
\label{a6}
Given any $a,b\in\M_0(\C)$, we have the following.
\begin{enumerate}
	\item[\textup{(1)}] $aba^*=0$ if and only if at least one of $ab,ba^*$ is zero.
	\item[\textup{(2)}] $\langle a,b\rangle=0$ if and only if at least one of $ab,ba^*,a^*b^*$ is zero.
\end{enumerate}
\end{lemma}

\begin{proof}
For both parts (1) and (2), one implication is clear: if one of $ab,ba^*,a^*b^*$ is zero, then $\langle a,b\rangle=0$, and if one of $ab,ba^*$ is zero, then $aba^*=0$. We shall now prove the converses. Since the lemma is trivial if $a$ or $b$ is zero, we may assume $a,b\in\M_0(\C)$ are both nonzero. Without loss of generality, we shall assume after conjugation that $a$ is in Jordan normal form, so that
$$
a=\begin{bmatrix}\lambda & 0\\ 0 & 0\end{bmatrix}\text{ for some $\lambda\in\C^\times$},\quad\text{or}\quad a=\begin{bmatrix}0 & 1\\ 0 & 0\end{bmatrix}.
$$
In the former case, i.e.~when $\tr(a)=\lambda\neq0$, writing $b=(b_{ij})$ we have
$$ab=\begin{bmatrix}\lambda b_{11} & \lambda b_{12}\\ 0 & 0\end{bmatrix},\quad ba^*=\begin{bmatrix}0 & \lambda b_{12}\\0 & \lambda b_{22}\end{bmatrix},\quad a^*b^*=\begin{bmatrix}0 & 0\\-\lambda b_{21} & \lambda b_{11}\end{bmatrix},$$
and
$$aba^*=\begin{bmatrix}0 & \lambda^2b_{12}\\0 & 0\end{bmatrix},\quad \langle a,b\rangle=\begin{bmatrix}-\lambda^2 b_{12}b_{21} & \lambda^2 b_{12}b_{11}\\ 0 & 0\end{bmatrix}.$$
In the latter case, i.e.~when $\tr(a)=0$, writing $b=(b_{ij})$ we have
$$ab=\begin{bmatrix}b_{21} & b_{22}\\ 0 & 0\end{bmatrix},\quad ba^*=\begin{bmatrix}0 & -b_{11}\\0 & -b_{21}\end{bmatrix},\quad a^*b^*=\begin{bmatrix}b_{21} & -b_{11}\\0 & 0\end{bmatrix},$$
and
$$aba^*=\begin{bmatrix}0 & -b_{21}\\0 & 0\end{bmatrix},\quad\langle a,b\rangle =\begin{bmatrix}b_{21}^2 & -b_{21}b_{11}\\ 0 & 0\end{bmatrix}.$$
We now proceed with our proof.

(1) Suppose that $ab$ and $ba^*$ are nonzero yet $aba^*=0$. The condition $aba^*=0$ would imply that $b_{12}=0$ if $\tr(a)\neq0$, and $b_{21}=0$ if $\tr(a)=0$. In both cases, since $ab,ba^*\neq0$ we must have $b_{11},b_{22}\neq0$, contradicting the assumption that $\det(b)=0$. Thus, we must have $aba^*=0$.

(2) Assume toward contradiction that $ab,ba^*,a^*b^*$ are all nonzero yet $\langle a,b\rangle=0$. Consider the case $\tr(a)=\lambda\neq0$. The condition $\langle a,b\rangle=0$ would imply that $b_{12}=0$ or $b_{21}=b_{11}=0$. If $b_{12}=0$, then since $ab,ba^*\neq0$ we must have $b_{11},b_{22}\neq0$, contradicting the assumption that $\det(b)=0$. If $b_{21}=b_{11}=0$, then $a^*b^*=0$ contradicting our assumption. Thus, we must have $\langle a,b\rangle\neq0$.

It remains to treat case $\tr(a)=0$. The condition $\langle a,b\rangle=0$ would imply that $b_{21}=0$. But if $b_{21}=0$, then since $ab,ba^*\neq0$ we must have $b_{11},b_{22}\neq0$, contradicting the assumption that $\det(b)=0$. Thus, we must have $\langle a,b\rangle\neq0$, as desired. This finishes the proof.
\end{proof}

\begin{lemma}
\label{a8}
The locus of $(a,b)\in\M_0^2(\C)$ such that at least two of $ab$, $ba^*$, and $a^*b^*$ are zero has dimension at most $4$.
\end{lemma}

\begin{proof}
Since the locus where $a=0$ or $b=0$ has dimension at most $3$, we may restrict our attention to the locus with $a,b\neq0$.
\begin{enumerate}
	\item[\textup{(1)}] Consider the locus $ab=ba^*=0$. We then have $ab^*=(ba^*)^*=0$ and hence $\tr(b)a=ab+ab^*=0$ and hence $\tr(b)=0$, since $a\neq0$ by assumption. Thus, $b$ varies over a locus of dimension $2$. For each fixed value of $b$, the condition $ab=0$ implies that $a$ varies over a locus of dimension $2$ by Lemma \ref{a5}. Hence, the locus where $ab=ba^*=0$ has dimension at most $2+2=4$.
	\item[\textup{(2)}] Consider the locus $ab=a^*b^*=0$. We then have $ba=(a^*b^*)^*=0$ and hence $[a,b]=ab-ba=0$. By Lemma \ref{a2}, for fixed $a$ (assumed nonzero) we see that $b$ varies over a locus of dimension at most $1$. Hence, the locus where $ab=a^*b^*=0$ has dimension at most $3+1=4$.
	\item[\textup{(3)}] Consider the locus $ba^*=a^*b^*=0$. We then have $ba=0$ and $ab^*=0$. Replacing the role of $b$ and $a$, we reduce to case (1), showing that the said locus has dimension at most $4$.
\end{enumerate}
This concludes the proof of the lemma.
\end{proof}

\begin{lemma}
\label{a7}
The morphism:
\begin{enumerate}
	\item[\textup{(1)}] $\M_{0,0}^2\to\M_0$ given by $(a_1,a_2)\mapsto a_1a_2$ is surjective over $\M_0\setminus\M_{0,0}$. The preimage of this morphism over $\M_{0,0}\setminus\{0\}$ is empty.
	\item[\textup{(2)}] $(\M_0\setminus\M_{0,0})\times\M_{0,0}\to\M_0$ given by $(a_1,a_2)\mapsto a_1a_2$ is surjective.
	\item[\textup{(3)}] $\langle-,-\rangle:\M_0^2\to\M_0$ is surjective.
\end{enumerate}
\end{lemma}
\begin{proof}
(1) By Lemma \ref{a1}, Given any $b\in\M_0(\C)$, there exist $a_1,a_2\in\M_{0,0}(\C)$ such that $\tr(a_1a_2)=\tr(b)$. If $\tr(b)\neq0$, then this implies that $(ga_1g^{-1})(ga_2g^{-1})=b$ for some $g\in\SL_2(\C)$, and $ga_1g^{-1},ga_2g^{-1}\in\M_{0,0}(\C)$. For the last statement, note first that
$$\begin{bmatrix}0 & 1\\ 0 & 0\end{bmatrix}\begin{bmatrix}x & y\\ z & w\end{bmatrix}=\begin{bmatrix}z & w\\ 0 & 0\end{bmatrix}$$
and, provided that $\left[\begin{smallmatrix}x & y\\ z & w\end{smallmatrix}\right]\in\M_{0,0}(\C)$, if the right hand side has trace zero then it must in fact be zero. Since any pair $(a_1,a_2)\in\M_{0,0}^2(\C)$ with $a_1$ nonzero is conjugate to a pair of the form $(\left[\begin{smallmatrix}0 & 1\\ 0 & 0\end{smallmatrix}\right],\left[\begin{smallmatrix}x & y\\ z & w\end{smallmatrix}\right])$, the last statement follows.

(2) The same argument as in the proof of part (1) goes through when $b\in\M_0(\C)$ satisfies $\tr b\neq0$ or $b=0$. Suppose that $\tr b=0$ and $b\neq0$. Up to conjugation by $\SL_2$, we may assume that $b=\left[\begin{smallmatrix}0 & 1\\ 0 & 0\end{smallmatrix}\right]$. Note then that we have
$$\begin{bmatrix}1 & 1\\ 0 & 0\end{bmatrix}\begin{bmatrix}0 & 1\\0 & 0\end{bmatrix}=\begin{bmatrix}0 & 1\\ 0 & 0\end{bmatrix}.$$
This proves that the morphism $(\M_0\setminus\M_{0,0})\times\M_{0,0}\to\M_0$ is surjective, as desired.

(3) By Lemma \ref{a1}, given any $b\in\M_0(\C)$ there exist $a_1,a_2\in\M_0(\C)$ such that $\tr\langle a_1,a_2\rangle=\tr b$. If $\tr b\neq0$, then $\langle ga_1g^{-1},ga_2g^{-1}\rangle=b$ for some $g\in\SL_2(\C)$. Consider now $\tr b=0$. We have $\langle0,0\rangle=0$. Furthermore, we have
$$\left\langle\begin{bmatrix}1 & -1\\ 0 & 0\end{bmatrix},\begin{bmatrix}1 & 1\\ 0 & 0\end{bmatrix}\right\rangle=\begin{bmatrix}1 & -1\\ 0 & 0\end{bmatrix}\begin{bmatrix}1 & 1\\ 0 & 0\end{bmatrix}\begin{bmatrix}0 & 1\\ 0 & 1\end{bmatrix}\begin{bmatrix}0 & -1\\ 0 & 1\end{bmatrix}=\begin{bmatrix}0 & 2\\ 0 & 0\end{bmatrix}$$
which shows that, letting $a_1$ and $a_2$ respectively be the two matrices on the left hand side, there exists $g\in\SL_2(\C)$ such that $\langle ga_1g^{-1},ga_2g^{-1}\rangle=b$, as desired.
\end{proof}

\end{document}